\documentclass[12pt]{amsart}
\usepackage[colorlinks=true, pdfstartview=FitV, linkcolor=blue, citecolor=blue, urlcolor=blue]{hyperref}
\usepackage{amssymb}
\calclayout

\def\C{\mathbb {C}}
\def\R{\mathbb {R}}

\def\N{\mathbb {N}}
\def\Z{\mathbb {Z}}
\def\O{\mathcal O}
\def\lie#1{\mathfrak{ #1}}
\def\lieh{\lie h}
\def\lieg{\lie g}
\def\liek{\lie k}

\def\lieu{\lie u}
\def\inv{^{-1}}

\newcommand{\rank}{\operatorname{rank}}
\newcommand{\SL}{\operatorname{SL}}
\newcommand{\Sp}{\operatorname{Sp}}
\newcommand{\SO}{\operatorname{SO}}

\newcommand{\PGL}{\operatorname{PGL}}
\newcommand{\Ker}{\operatorname{Ker}}
\newcommand{\pr}{\operatorname{pr}}
\def\Lie{\operatorname{Lie}}
\def\GL{\operatorname{GL}}
\def\Hom{\operatorname{Hom}}
\def\End{\operatorname{End}}

\def\Ad{\operatorname{Ad}}
\def\phi{\varphi}

\def\rank{\operatorname{rank}}
\def\U{\operatorname{U}}

\def\quot#1#2{#1/\!\! /#2}

\def\ss{\mathcal S}
\def\AA{\mathsf{A}}
\def\CC{\mathsf{C}}
\def\BB{\mathsf{B}}
\def\DD{\mathsf{D}}
\def\EE{\mathsf{E}}
\def\GG{\mathsf{G}}

\def\pt{\partial}
\def\Aff{\operatorname{Aff}}

\def\Dbar{\leavevmode\lower.6ex\hbox to 0pt{\hskip-.23ex
    \accent"16\hss}D}
    \def\sspan{\operatorname{span}}

\numberwithin{equation}{subsection}

\newtheorem{theorem}[subsection]{Theorem}
\newtheorem{lemma}[subsection]{Lemma}
\newtheorem{proposition}[subsection]{Proposition}
\newtheorem{corollary}[subsection]{Corollary}

\theoremstyle{definition}

\newtheorem{definition}[subsection]{Definition}

\theoremstyle{remark}

\newtheorem{remark}[subsection]{Remark}

\newtheorem{example}[subsection]{Example}

\title[Linear maps preserving orbits]{\boldmath Linear maps preserving orbits} 
 \author{Gerald W. Schwarz}
\address{Department of Mathematics\\
Brandeis University\\
Waltham, MA 02454-9110}
\email{schwarz@brandeis.edu}

\begin{document}
\begin{abstract}
Let $H\subset\GL(V)$ be a connected complex reductive group where $V$ is a finite-dimensional complex vector space. Let  $v\in V$ and let $G=\{g\in\GL(V)\mid gHv = Hv\}$. Following  Ra\"is \cite{Rais1} we say that the orbit $Hv$ is \emph{characteristic for $H$\/} if the identity component of $G$ is  $H$.  If $H$ is semisimple, we say that $Hv$ is \emph{semi-characteristic\/} for $H$ if the identity component of $G$ is an extension of $H$  by a torus. We classify the $H$-orbits which are not (semi)-characteristic in many cases.
\end{abstract}

\maketitle
\section{Introduction}

Let $K$ be a field. Then $H:=\PGL_n(K)$ acts on $V:=M(n,K)$ via conjugation. There is a large literature on solving  \emph{linear preserver problems\/}, that is, on finding the subgroups of $\GL(V)$ which preserve a certain set $F$ of $H$-orbits in $V$. See \cite{Pierce} for a survey. One method of solving such problems is to classify all possible subgroups of $\GL(V)$ containing $H$ and then check to see if these subgroups preserve $F$. This idea goes back at least to Dynkin \cite{Dynmax} and has been used in many papers, e.g., 
\cite{Gur1, Gur2, Gur3, Dok1, Dok2, Dok3}. We generalize the problem  (but only in characteristic zero) by letting $H$ be a reductive complex  algebraic group,  letting $V$ be an arbitrary finite dimensional representation of $H$ and letting $F$ be an $H$-orbit $H v$. The question then becomes: What is the subgroup $G$ of $\GL(V)$ which preserves $H
 v$?  The method of solution is often to look at the possible $G$ and possible $G_v$ such that  $G=HG_v$ (which implies  that $Gv=Hv$). We are able to answer the question in many circumstances. We are particularly interested in identifying those cases where $G^0$ is the image of $H$, which, in the language of Ra\"is \cite{Rais1}, means identifying those $H$-orbits which are \emph{characteristic}.

Our base field is $\C$, the field of complex numbers.  Let $V$ be a finite dimensional $H$-module where $H$ is a  connected reductive group.  Let $0\neq v\in V$ and set $G:=\{g\in\GL(V)\mid gHv=Hv\}$. Then $G$ is a closed algebraic subgroup of $\GL(V)$ (see \ref{lem:closed} below), We say that \emph{$Hv$ is characteristic for $H$\/} (or simply that $v$ is characteristic for $H$ or just that $v$ is characteristic) if $G^0$ is the image of $H$ in $\GL(V)$. (From now on we will not distinguish $H$ from its image in $\GL(V)$, so we will say that $v$ is characteristic if $G^0=H$, even though this is not quite correct.)\  The definition that $Hv$ is \emph{semi-characteristic\/} is as above, except that we require only that $G^0$ is an extension of $H$ by a torus   (so $G$ has to be reductive).  
In general,  $G$ is not reductive (see Examples \ref{ex:glpq},  \ref{ex:exceptions}, \ref{ex:gu2=0} and \ref{ex:complicated}). We say that $v$ is \emph{almost characteristic\/} if $H$ is a Levi factor of $G^0$ and that \emph{$v$ is almost semi-characteristic\/} if $H$ contains the semisimple part of a Levi factor of $G^0$.

In \S 2 we consider some elementary properties of our definitions. We see that one has a chance for $G^0=H$ only in the case that $v\in V$ is \emph{generic\/}, which is equivalent to saying that $H
 v$ spans $V$. In \S 3 we consider what can happen to $G$ if we add a trivial factor to $V$. We show that $Hv$ is characteristic if $H$ is a torus and $v\in V$ is generic. In \S 4 we consider the case that $H$ is simple of rank at least 2 and $V$ is irreducible. We recall some fundamental results of A. Onishchik which apply. We are then able to classify the irreducible $H$-modules $V$ and $v\in V$ such that $Hv$ is not semi-characteristic. We determine which orbits are semi-characteristic in the adjoint representation of a semisimple group. In \S 5 we consider the case that $H$ is simple of rank at least 2 and $V$ is reducible. We determine the possible semisimple $G$ containing $H$ such that $Gv=Hv$ for  $v\in V$.   In \S 6 we consider the case that $H$ is semisimple and $V$ is irreducible. In \S 7 we determine the structure of $G$ when $H=\SL_2$. In an appendix we prove branching rules which we need to establish our results.
 
Our  thanks go to  M. Ra\"is for his questions and conjectures in   \cite{Rais1} which led to this paper. We thank Peter Heinzner, Peter Littelmann, Ernest Vinberg and Arkady Onishchik for helpful remarks and  Alfred No\"el and Steven G. Jackson for help with calculations. We thank the University of Poitiers, the Ruhr-Universit\"at Bochum and the Mathematisches Institut, Universit\"at zu K\"oln for their warm hospitality while this paper was being written. Finally, special thanks to the referee for a very meticulous reading of the manuscript and many helpful remarks. (S)he also found a serious error in our original version of the section on $\SL_2$.

\section{Elementary remarks}
 We consider when we can remove the prefixes ``almost'' and ``semi.''   We also   reduce to the case that  $Hv$ spans $V$. First we show that $G$ is closed in $\GL(V)$.
 
 \begin{lemma} \label{lem:closed} Let $V$ be a finite-dimensional $H$-module where $H$ is algebraic. Let $G=\{g\in\GL(V)\mid gHv=Hv\}$. Then $G$ is a closed subgroup of $\GL(V)$.
 \end{lemma}
 
 \begin{proof}
 Let $G_1=\{g\in \GL(V)\mid g\overline{Hv}=\overline{Hv}\}$ and $G_2=\{g\in \GL(V)\mid g(\overline{Hv}\setminus Hv)=(\overline{Hv}\setminus Hv)\}$. Then $G_1$ and $G_2$ are closed subgroups of $\GL(V)$ and $G=G_1\cap G_2$.
 \end{proof}
 
 We now consider complexifications of compact group actions. Let $C$ be a compact Lie group and $W$ a real $C$-module. Let $w\in W$ and assume that $Cw$ spans $W$.
 \begin{proposition} \label{prop:compact} Let $C$, $W$ and $w$ be as above.
 Let $L=\{g\in\GL(W)\mid gCw=Cw\}$. Then $L$ is compact.
 \end{proposition}
 \begin{proof}
 Fix a basis $w_1,\dots,w_n$ of $W$ lying in $Cw$ and let $||\cdot||$ be a norm on $W$. Then for $g\in L$ and $1\leq i\leq n$, $||gw_i||$ is bounded by a constant  which is independent of $g$. Thus $L$ is a closed bounded subset of $\GL(W)$, hence compact.
 \end{proof}
 
 \begin{corollary} 
 Let $H=C_\C$ be the complexification of $C$ acting on $V=W\otimes_\R\C$. Let $G=\{g\in\GL(V)\mid gHw=Hw\}$. Then $G$ is the complexification of $L$, hence reductive.
 \end{corollary}
 
 \begin{proof}
 Since $Cw$ is real algebraic \cite[Lemma 4.3]{SchAlgebraicquotients}, it is defined by an ideal $I\subset \R[W]$, and clearly the complex zeroes of $I$ are $Hw$. Let $I_s$ denote the subspace of $I$ of elements of degree at most $s$, $s\in \N$. Then $I$ is generated by some $I_s$. Let $f_1,\dots,f_m$ be a basis for $I_s$. Then $g\in\GL(W)$ lies in $L$ if and only if $g^*f_i\in I_s$ for all $i$. This gives a set of real equations defining the compact Lie group $L$, and   the complex solutions of these equations are $L_\C$. But the complex solutions of the equations are clearly $G$. Thus $G=L_\C$.
 \end{proof}
 
 Recall (\cite[Ch.\ II Theorem 11]{JacobsonLA, JacobsonLA2}) that if $G\subset\GL(V)$ acts irreducibly on $V$, then $G$ is reductive. 
 \begin{corollary}
 Let $H$ be reductive, let $V$ be an $H$-module and let $v\in V$. Suppose that $v$ is almost semi-characteristic for $H$. Then $v$ is semi-characteristic in the following two cases.
 \begin{enumerate}
\item $V$ is an irreducible representation of $H$.
\item There is a compact Lie group $C$ and real $C$-module $W$ such that $V=W\otimes_\R\C$, $v\in W$  and $H=C_\C$.
\end{enumerate}
 \end{corollary}

The following result characterizes when $Hv$ is a cone.

\begin{proposition} Let $0\neq v\in V$ where $V$ is an $H$-module.
Suppose that $Hv$ is a cone. Then there is a $1$-parameter subgroup $\sigma\colon\C^*\to H$ such that $v$ is an eigenvector of $\sigma$ with nonzero weight.
\end{proposition}

\begin{proof}
Since $Hv$ is a cone, $v\in T_v(Hv)$ and there is an $X\in\lieh$ such that $X(v)=v$. Applying an element of $H$ we can   assume that $X\in\lie b$, a Borel subalgebra of $\lieh$. Write $X=s+n$ (Jordan decomposition) where $s$ is semisimple and $n$ is nilpotent. Then $s$ and $n$ are in $\lie b$. We can assume that $s\in\lie t\subset \lie b$ where $\lie t$ is the Lie algebra of $T$, a maximal torus of $H$. Write $v=\sum_{\lambda \in\Lambda}v_\lambda$ as a sum of nonzero weight vectors where $\Lambda$ is the set of weights of $V$ relative to $T$ such that $v_\lambda\neq 0$. Let $\Phi$ be the set of positive roots. Then for $\lambda\in\Lambda$ we have $(s+n)v_\lambda=sv_\lambda$ modulo $\sum_{\mu>\lambda} V_\mu$ where $\mu>\lambda$ means that $\mu\in\lambda+\N\Phi$. Thus by an easy induction we get that $sv_\lambda=v_\lambda$ for all $\lambda\in\Lambda$ so that $sv=v$. Hence $S:=\{t\in T\mid tv\in\C^* v\}^0$ is a subtorus of $T$ which acts nontrivially on $v$. It follows that there is a one-parameter subgroup  $\sigma\colon \C^*\to S$ as desired.
\end{proof}

\begin{proposition}
Let $0\neq v\in V$ where $V$ is an irreducible $H$-module. Suppose that $\C^*v\not\subset Hv$. Then  $v$ is characteristic if it is semi-characteristic. In particular, this holds if $v$ is not in the null cone of $V$.
\end{proposition}
\begin{proof}
The group $G$ is reductive and its center is contained in the scalar matrices. Under our hypotheses on $v$, the center must be finite.
\end{proof}
 
  Let $V=\bigoplus_{i=1}^k n_i V_i$ be  the isotypic decomposition of an $H$-module where $H$ is nontrivial  reductive. Let $v\in V$. Then $v=(v_{ij})$ where $v_{ij}$ belongs to the $j$th copy of $V_i$, $j=1,\dots,n_i$, $i=1,\dots,k$. Let $S$ denote $\GL(V)^H=\GL(n_1)\times\dots\times \GL(n_k)$. Let  $U_i\subset V_i$ be the linear subspace of $V_i$ generated by the $v_{ij}$, $j=1,\dots,n_i$. If $\dim U_i=n_i$ for all $i$, then we say that $v$ is \emph{generic}. 

\begin{proposition}\label{prop:generic}
Let $H$ be reductive, let $V$ be an $H$-module and let $v\in V$. Then the following are equivalent.
\begin{enumerate}
\item The span of $Hv$ is   $V$.
\item There is no nontrivial one parameter  subgroup of $S=\GL(V)^H$ which fixes $v$.
\item The vector $v$ is  generic.
\end{enumerate}
\end{proposition}

\begin{proof} If $s\in S$, then $v$ satisfies one of the conditions if and only if $sv$ does. Clearly, if (1) or (3) fails, we can find an $s$ and an $i$ such that $(sv)_{i1}=0$. If $W$ denotes the first copy of $V_i$ in $n_iV_i$, we have that $\C^*=GL(W)^H\subset \GL(V)^H$ is a one-parameter subgroup fixing $v$, so (2) fails. Conversely, if (2) fails, the fixed point set of a ``bad'' one-parameter subgroup is a proper $H$-submodule of $V$ containing $v$ and (1) and (3) fail.
\end{proof} 

  \begin{corollary}
  Let $v\in V$ and let $G$ be a Levi component of $\{g\in\GL(V)\mid gHv=Hv\}$. Then $v$ is generic for $H$ if and only if it is generic for $G$.
  \end{corollary}

 If $v$ is not generic, then $G$ is in a rather trivial way larger than $H$. To avoid this case, we usually assume from now on that $v$ is generic.

\section{Trivial factors and tori}

Let $v\in V$  and  suppose that $V^H= (0)$. Consider $\tilde v=(v,1)\in \tilde V:=V\oplus\C$. Let $\tilde G=\{g\in \GL(\tilde V)\mid gH\tilde v=H\tilde v\}$. We conjecture that $\tilde G=G$, where $G\subset\GL(\tilde V)$ in the canonical way. Equivalently, we conjecture that the subgroup of the affine group $\Aff(V)$ preserving $Hv$ lies in $\GL(V)$. Note that $v$ generic implies that $\tilde v$ is generic (we can add at most a one-dimensional fixed point set). The following example shows that the conjecture fails if  $H$ is not reductive.

 \begin{example}\label{ex:nonred}
 Let $H=(\C,+)$ act on $\C^2$ by sending $(a,b)\in\C^2$ to $(a,ta+b)$, $t\in H$. Let $\tilde H=H\times\C$ where $(t,s)\cdot (a,b)=(a,ta+s+b)$, $(t,s)\in \tilde H$, $(a,b)\in\C^2$. Then for $a\neq 0$, the $H$ and $\tilde H$ orbits of $(a,b)$ are the same, where $H\subset\GL(\C^2)$, $\tilde H\subset\Aff(\C^2)$ and $\tilde H\not\subset\GL(\C^2)$.
 \end{example}

For this section only $G$ will denote the subgroup of $\Aff(V)$ preserving $Hv$ (rather than the corresponding subgroup of $\GL(V)$). It is easy to see that we can always reduce to the case that $V^H=(0)$, which we assume holds for the rest of this section.

We have a homomorphism  $\Aff(V)\to\GL(V)$ which sends an element $(g,c)\in G\subset\GL(V)\ltimes V$ to $g\in\GL(V)$. Let $G'$ denote the image of $G$ in $\GL(V)$.

\begin{lemma}
The homomorphism $G\to G'$ is injective.
\end{lemma}
  
 \begin{proof}
The kernel $K$ of $G\to G'$ consists of the pure translations in $G$, i.e., the homomorphisms $x\mapsto x+c$ where $x$, $c\in V$. Clearly $K$ is isomorphic to a closed subgroup of the additive group $(V,+)$ of $V$. Now $(V,+)$ has Lie algebra $V$ (trivial bracket) and the exponential map is the identity. Thus $\liek$ is a vector subspace $W$ of $V$ and $K/K^0$ is isomorphic to a finite subgroup of $(V/W,+)$.  Hence $K$ is connected and $K=(W,+)$ where $W$ must be $H$-stable. Let $\pi\colon V\to W$ be an $H$-equivariant projection (here we use that $H$ is reductive). Then there are elements of $G$ which translate $v$ to $v'$ where $\pi(v')$ is arbitrary. Since $H$ preserves $W$ and $\Ker \pi$, this is not possible for elements of $H$, unless $W=0$. Hence $K$ is the trivial group.  
  \end{proof}

  Note that injectivity fails in the case of Example \ref{ex:nonred}.
  
  \begin{lemma}\label{lem:conj}
Let $M$ be a reductive subgroup of the affine group $\Aff(V)$. Then there is an $\alpha\in\Aff(V)$ such that $\alpha M\alpha\inv\subset\GL(V)$.
\end{lemma}

\begin{proof}
We use transcendental methods. Choose a hermitian metric on $V$ so that we have a unitary group $\U(V)\subset \GL(V)$. Let $K$ be a maximal compact subgroup of $M$. Then $M$ is the complexification $K_\C$ of $K$. Now any compact subgroup of $\Aff(V)$ is contained in  a maximal compact subgroup of $\Aff(V)$ and all the maximal compact subgroups of $\Aff(V)$ are conjugate \cite[Ch.\ XV Theorem 3.1]{Hochschildstructure}. But clearly  $\U(V) \subset\Aff(V)$ is maximally compact. Thus $K$ is conjugate to a subgroup of $\U(V)$, hence $M$ is conjugate to a subgroup of $\U(V)_\C=\GL(V)$.
\end{proof}

  \begin{proposition}\label{prop:easy}
  In the following cases  $G\subset\GL(V)$.
  \begin{enumerate}
\item The image $G'\subset \GL(V)$ is reductive.
\item There is an $h'\in H$ such that $h'v=\lambda v$, $\lambda\in\C$, $\lambda\neq 1$.
\end{enumerate}
\end{proposition}
 
  \begin{proof} 
   If (1) holds, then $G$ is reductive and there is an element $\alpha\in\Aff(V)$ such that $\alpha G\alpha\inv\subset\GL(V)$, hence  $\alpha H\alpha\inv\subset \GL(V)$. But one easily sees that any affine transformation conjugating $H$ into $\GL(V)$ must have translation part which is fixed by $H$. But $V$ contains no nonzero $H$-fixed vectors. Hence $G\subset\GL(V)$.
  
  Assume (2). Let $x\mapsto c+A(x)$ be an element  of $\lieg=\Lie(G)$ where $0\neq c\in V$ and $A\in\lie{gl}(V)$. Then the difference  of $c+A(hv)$ and $c+A(hh'v)$ is a nonzero multiple of  $A(hv)$ and lies in $\lieh(hv)$ for any $h\in H$. Thus $A$ itself lies in $\lieg$ and $\lieg$ contains the linear and translation parts of its elements. But $\lieg$ cannot contain pure translations, as we saw above. Thus  $\lieg\subset\lie{gl}(V)$ and $G\subset\GL(V)$.
\end{proof}

\begin{corollary} 
We have that $G\subset\GL(V)$ in the following cases.
\begin{enumerate}
\item $V$ is an irreducible $H$-module.
\item $V$ is an $\SL_2$-module  whose irreducible components are all of even dimension, i.e., a module all of whose weights are odd.
\end{enumerate}
\end{corollary}

\begin{remark}
Suppose that $C$, $W$, $w\in W$ are as in Proposition \ref{prop:compact} where $W^C=0$. Let $L$ denote the subgroup of the real affine group of $W$ stabilizing $Cw$. Then  one can show that $L$   is compact, and as above one sees that $L\subset\GL(W)$. Complexifying, we see that the subgroup of the affine group of $V=W\otimes_\R\C$ preserving $Hw$, where $H=C_\C$,   is again just the complexification of $L$, a subgroup of $\GL(V)$. 
\end{remark}

\begin{proposition}\label{prop:vij}
Let $V=\oplus_i n_iV_i$ be the isotypic decomposition of $V$. Suppose that for no $i$  and $j$ do we have that $V_i$ occurs in $\Hom(V_j,V_i)$. Then $G\subset\GL(V)$.
\end{proposition}

\begin{proof}
Suppose that $G\not\subset\GL(V)$. Then we would have a subspace of $\lieg$ consisting of elements $A_w+w$, $w\in W$, where $W\simeq V_i$ is an irreducible submodule of $V$, $A_w\in\lie{gl}(V)$ and $hA_wh\inv=A_{hw}$ for $h\in H$.  Our hypotheses imply that  $A_w$ followed by projection to $n_iV_i$ is zero.  Thus $\exp(A_w+w)(v)$ has the same projection to  $W$  as $v+w$. Hence we cannot have $Hv=Gv$. 
\end{proof}

 \begin{example}
Let  $V:=\sum_{i=1}^n m_i\phi_i$ and $H=\AA_n$, $n\geq 1$, where $\phi_i$ is the $i$th fundamental representation of $H$, $i=1,\dots,n$. Then $G\subset\GL(V)$.
 \end{example}
  
   \begin{theorem}\label{thm:torus}
Let $H$  be a torus. Then 
\begin{enumerate}
\item $G\subset\GL(V)$.
\item If $v\in V$ is generic, then $G^0=H$.
\end{enumerate}
  \end{theorem}
 
 \begin{proof}    We may assume that $V^H=(0)$. First consider (1) in the case that $H=\C^*$.
Let $W$ be the subspace of $V$ spanned by $H\cdot v$.  Then any $g\in G$ must preserve $W$, so we can replace $V$ by $W$.  Thus we can reduce to the case that $v\in V$ is generic. This implies that the weight spaces of $H$ are one-dimensional.
  We have a weight basis $v_1,\dots,v_n$ of $V$ such that $v=(v_1,\dots,v_n)$ where the weight of $v_i$ is $0\neq a_i\in\Z$. Suppose that the orbit of $v$ is preserved by a transformation $(g,c)$ where $(g,c)(v)=(\sum_{j}a_{ij}v_j+c_i)$. Here  the $a_{ij}$ and $c_i$ are scalars. Then the $i$th component of $g(\lambda\cdot v)$ (where $\lambda$ is a parameter in $H=\C^*$) is $\sum_{j}a_{ij}\lambda^{a_j}v_j+c_i$. Now the powers of $\lambda$ that occur are distinct, hence the Laurent polynomial in $\lambda$ that gives the $i$th component has some nonzero coefficient for a nonzero power of $\lambda$. If $c_i\neq 0$, then one can see that the polynomial takes on the value 0 for some $\lambda\neq 0$. But the $\C^*$-orbit of $v$ is nonzero in the $i$th slot. Thus $c_i=0$ for all $i$ and $g$ lies in $\GL(V)$ so we have (1).  The reasoning above also shows that for each $i$ there is a unique $j$ such that $a_{ij}\neq 0$. Thus a power $g^k$ of $g$ preserves the weight spaces. Then $g^kv=hv$ for some $h\in H$, and it follows that $g^k=h$. Thus we have (2). Note that $g$ normalizes $H=\C^*$, so that we actually have $g^2\in H$.
  
  Now suppose that $H$ is a torus. As before, to prove (1), we can assume that $v$ is generic. Let $(g=(a_{ij}),c)\in G$. Choose a 1-parameter subgroup $\lambda$ of $H$ such that all the characters of $V$, restricted to $\lambda$, are distinct. It follows, as above, that $c=0$ and that a power of $g$ lies in $H$.
 \end{proof}

\subsection{}\label{subsec:nilpotent}  Let $G_0$ denote a Levi component of $G$ containing $H$. Then as we saw before, we must have that $G_0\subset\GL(V)$. We can write $G'$ as $G_0\ltimes G'_u$ where $G'_u$ is the unipotent radical of $G'$. Then we have the corresponding decompostion of $\lieg'$ as $\lieg_0\ltimes\lieg'_u$. As $H$-module, $\lieg'_u$ is completely reducible. Assuming that $G$ is not contained in $\GL(V)$  we can choose an irreducible $H$-module $W\subset\lieg'_u$ whose inverse image in $\lieg$ is  not contained in $\lie{gl}(V)$. Then we have a copy of $W$ in $V$ and elements $A_w\in\lie{gl}(V)$, $w\in W$, such that $x\mapsto A_w(x)+w$ lies in $\lieg$ and $\{A_w\}_{w\in W}$ maps to our copy of $W$ in $\lieg'_u$. For all $h\in H$ we have $hA_wh\inv=A_{hw}$.

 \begin{theorem}\label{thm:nullcone}
 Suppose that $v\in V$ is generic and  in the null cone. Then $G\subset\GL(V)$.
 \end{theorem}
 
 \begin{proof}
Suppose the contrary. Let $V=\oplus_i n_i V_i$ be the isotypic decomposition of $V$ as $H$-module. Let $A_w+w\in\lieg$, $w\in W$ be as above where we may assume that $W=V_i$ (first copy) for some $i$. Let $\pi\colon V\to W$ be an equivariant projection and set  $v_i=\pi(v)$. Since $v$ is generic, $v_i\neq 0$. The projection of  $\exp(w+A_w)(v)$ to $W$ has the form $v_i+w+p(v,w)$ where $p(v,w)$ is a polynomial which has no linear factors in $w$ and such that the coefficients of the various monomials in $w$ are polynomials in $v$ without constant term. By applying elements $h\in H$ we can make the coefficients of $hw$ in $p(hv,hw)$ as small as we want. But there is no loss if we replace $hw$ by $w$ since we are able to consider all possible $w$. Thus we can assume that the coefficients of the monomials in $w$ in $p(v,w)$ are very small, in which case the inverse function theorem tells us that $w\mapsto w+p(v,w)$ covers a ball around $0\in W$ whose radius we can choose to be independent of $v$ (for $v$ close to zero). Then we see that $w\mapsto v_i+w+p(v,w)$ takes on the value $0$. Thus $Gv$ contains a point which projects to $0\in W$, which is impossible. Hence $G\subset\GL(V)$.
 \end{proof}
 
 Recall that $V$ is called \emph{stable\/} if it contains a nonempty Zariski open subset of closed orbits. 
 
 \begin{corollary}
 Let $V$ be stable with a one-dimensional quotient. Then $G\subset\GL(V)$.
 \end{corollary}

\begin{proof}
We have that $\C[V]^H=\C[f]$ where $f$ is homogeneous of degree $d>1$. Moreover, $f\inv(f(v))=Hv=Gv$ if $f(v)\neq 0$. Now the case that $f(v)=0$ follows from Theorem \ref{thm:nullcone} and if $f(v)\neq 0$, then $Gv\supset \Gamma v$ where $\Gamma\subset G$ is a finite subgroup isomorphic to $\Z/d\Z\subset\C^*$ acting via scalar multiplication on $V$. Then $G\subset\GL(V)$ by Proposition \ref{prop:easy}(2).
\end{proof}
 \begin{remark}
 A case by case check  shows that $H$ simple and $\dim \quot VH=1$ implies that $G\subset\GL(V)$.
 \end{remark}
 
 \begin{theorem}\label{thm:affsl2}
 If $H=\SL_2$, then $G\subset\GL(V)$.
 \end{theorem}
 
We prove the theorem by contradiction, so assume that we have   $A_w+w$  as in \ref{subsec:nilpotent}. Then the $A_w$ lie  in a Lie algebra of nilpotent matrices, and by Engel's theorem we can find a partial flag $0=V_0\subset V_1\subset\dots\subset V_k=V$ such that $V_1$ is the joint  kernel of the $A_w$, $V_2/V_1\subset V/V_1$ is the joint kernel of   the $A_w$, etc. Note that the $V_j$ are $H$-stable.  
 
 \begin{lemma}\label{lem:1} 
 We have  $W\subset V_{k-1}$. 
 \end{lemma}
 
 \begin{proof} Since $A_w(v)\in V_{k-1}$, we have that $(A_w+w)(v+V_{k-1})=w+V_{k-1}$. Thus $\exp(A_w+w)(v+V_{k-1})=v+w+V_{k-1}$. Let $\pi$ be the projection of $V$ to $V/V_{k-1}$ with kernel $V_{k-1}$. If $W\not\subset V_{k-1}$, then  $\pi(Gv)$ contains a nontrivial $H$-stable subspace of $V/V_{k-1}$.    This is not possible for the $H$-orbit, hence $W\subset V_{k-1}$. 
 \end{proof}

 \begin{lemma}\label{lem:2} 
 Suppose that for some $j\geq 1$ we have  $W\subset V_{k-j}$ and $A_w(v)\in V_{k-j}$ for all  $w\in W$. Then the stabilizer of $v+V_{k-j}$ in $H$ is infinite.
 \end{lemma}
 
 \begin{proof}
 Suppose that $v+V_{k-j}$ has   finite stabilizer. Since $A_w(v)+w$ projects  to zero in $V/V_{k-j}$, we must have that $A_w(v)+w=0$, else the $G$-orbit of $v$ has dimension greater than $\dim H$. Now for $h\in H$, $A_w(hv)+w=h(A_{h\inv w}+h\inv w)(v)=0$, so that the average of $A_w(hv)+w$ over a maximal compact subgroup $K$ of $H$ is zero. Since $V^H=0$ and $A_w$ is linear, the average of $w+A_w(hv)$ over $K$ is $w$.  Thus $W=0$,  a contradiction. 
 \end{proof}
  
  \begin{lemma}\label{lem:3}
  Suppose that $v+V_{k-j}$ is in the null cone of $V/V_{k-j}$ for some  $j\geq 1$. Then $W\subset V_{k-j-1}$.
  \end{lemma}
  
  \begin{proof}
  Assume that $W\not\subset V_{k-j-1}$. Our argument in \ref{thm:nullcone} shows that the $G$-orbit of $v$ projected to the image of $W$ in $V/V_{k-j-1}$ contains zero, which is not possible for the $H$-orbit.
  Hence  $W\subset V_{k-j-1}$.
  \end{proof}

  \begin{lemma}\label{lem:4} Suppose that for some $j\geq 1$ we have that $A_w(v)\in V_{k-j}$ for all $w\in W$ and that $W\subset V_{k-j}$. Then   $v+V_{k-j}$ is in the null cone of $V/V_{k-j}$.
  \end{lemma}
  
  \begin{proof} Suppose not. Consider $V':=\C\cdot v+V_{k-j}\subset V$. Then for $z\in\C$,  $z(A_w+w)$ exponentiates to an element $g(z)$ of $\Aff(V')$ which fixes $v':=v+V_{k-j}\in V/V_{k-j}$. By Lemma  \ref{lem:2} we know that $v'$ has an infinite stabilizer $S$ in $H$. Since $v'$ is not in the null cone,  $S$ has identity component $T\simeq \C^*$.  For any $s\in S$ there is an $h_{z,s}\in H$   such that $h_{z,s}v=g(z)sv$. Then  $h_{z,s}\in S$ since $g(z)s$ fixes $v'$ so that the $g(z)$ preserve the $S$-orbit of $v$. The group generated by $T$ and the $g(z)$ is connected, so  it preserves the $T$-orbit of $v$.   By Theorem \ref{thm:torus}  we see that $w=0$, a contradiction.
   \end{proof}

  \begin{proof}[Proof of Theorem \ref{thm:affsl2}] We have that $A_w(v)\in V_{k-1}$ and $W\subset V_{k-1}$. Suppose that we have $A_w(v)\in V_{k-j}$  and $W\subset V_{k-j}$ for some $j\geq 1$. By Lemmas \ref{lem:2} and \ref{lem:4} and genericity of $v$ we may assume that $v+V_{k-j}$ is a sum of highest weight vectors.  By Lemma \ref{lem:3}, $W\subset V_{k-j-1}$, so that if $A_w(v)\in V_{k-j-1}$ we can continue. We eventually arrive at a case where $A_w(v)\not\in V_{k-j-1}$ (we cannot have a pure translation in $\lieg$). Since $v+V_{k-j}$ is a sum of highest weight vectors   there are unique elements $B_w\in\lie u$ such that $A_w(v)+w+B_w(v)\in V_{k-j-1}$. Here $\lieu$ is the Lie algebra of the standard unipotent subgroup of $H$. Since $A_w(v)+w\not\in V_{k-j-1}$, $B_w(v)\not\in V_{k-j-1}$ and  $v+V_{k-j-1}$ is not a sum of highest weight vectors. Since $v+V_{k-j}$ is a nonzero sum of highest weight vectors, the $H$-isotropy group of $v+V_{k-j-1}$, which is a subgroup of the $H$-isotropy  group of $v+V_{k-j}$, is finite.  Arguing as in Lemma \ref{lem:2} we obtain that $w':=A_w(v)+w+B_w(v)=0$ and that $W=0$, a contradiction. Hence $G\subset\GL(V)$.
    \end{proof}

From now on we will assume that $V^H=0$, even though we have only established our conjecture for $\SL_2$ or the case that $V$ is irreducible.

   \section{The case $H$ is simple and $V$ is irreducible}\label{sec:simple}
   
   Our goal in this section is to find the possible $G\subset\GL(V)$ preserving an orbit $Hv$ where $V$ is an irreducible $H$-module and $H$ is simple of rank at least two. We will see that perforce $G$ is simple. We begin by recalling some important results of Onishchik.
   
 Let $H\subset G$ where $G$  and $H$ are linear algebraic groups. Let  $V$ be an $H$-module. If $v\in V$ and $Gv=Hv$, then $G=HK$ where $K=G_v$. Conversely, $G=HK$ implies that  $Gv=Hv$ for $v\in V^K$. There is a rather restricted class of possibilities for $H$ and $K$ when $G$ is simple and $H$ is semisimple, as follows from the work of Onishchik \cite{Oncompact, Onreductive}. 
 
 If $K$ is a connected complex linear algebraic group, let $\liek$ denote its Lie algebra and let $L(K)$ denote a Levi subgroup of $K$. The next two theorems   follow  from \cite{Oncompact} and \cite{Onreductive} (see also \cite{Onbook} and \cite{OnGor}).
 
 \begin{theorem}
 Let $H$ and $K$ be connected algebraic subgroups of the connected reductive group $G$.  Then the following are equivalent.
 \begin{enumerate}
 \item $G=HK$.
 \item $G=\sigma(H)\tau(K)$ where $\sigma$ and $\tau$ are  any automorphisms of $G$.
 \item $G=L(H)L(K)$.
 \item $G_0=H_0K_0$ where $H_0$ and $K_0$ are maximal compact subgroups of $H$ and $K$ contained in a maximal compact subgroup $G_0$ of $G$.
 \item $\lieg=\lieh+\lie k$ (if $H$ and $K$ are reductive).
\end{enumerate}
\end{theorem}

  \begin{corollary}
  Suppose that $G=HK$ where all the groups are connected algebraic. Choose Levi factors $L(G)\supset L(H)$, $L(K)$.  Then $L(G)=L(H)L(K)$.
  \end{corollary}
 
 Now assume that  $\lieh$ and $\liek$ are reductive   subalgebras of the reductive Lie algebra $\lieg$. Let $\lieg_s$ be the sum of the simple components of $\lieg$ of rank at least 2 (the \emph{strongly semisimple\/} part of $\lieg$). Let $G_s$ be the corresponding subgroup of $G$. Let  $r(\lieg)$ be the sum of the center and simple components of rank 1 of $\lieg$ so that $\lieg=\lieg_s\oplus r(\lieg)$.
 
 \begin{theorem}\label{thm:g=h+k}
 Let $\lieh$ and $\liek$ be reductive subalgebras of the reductive Lie algebra $\lieg$. Then the following are equivalent.
 \begin{enumerate}
\item $\lieg=\lieh+\lie k$.
\item $\lieg_s=\lieh_s+\lie k_s$ and $r(\lieg)$ is the sum of the projections of $r(\lieh)$ and $r(\lie k)$ to $r(\lieg)$.
 \end{enumerate}
 \end{theorem}
 
 \begin{corollary}
Suppose that  $v\in V$ is not semi-characteristic and that $H$ contains a strongly semisimple subgroup. Then so does $G_v$.
 \end{corollary}
 
From the above and
 \cite{Oncompact} we have the following

 \begin{theorem}
 Let $G$ be connected, simple and simply connected of rank at least $2$. Let  $H$ and $K$ be connected semisimple subgroups of $G$ such that $G=HK$. Then, up to switching the roles of $H$ and $K$ and replacing  each of them by their image under an automorphism of $G$, all possibilities are listed in Table 1.
 \end{theorem}
  \begin{table}
\caption{}
 \begin{center}   
 \begin{tabular}{| l  | c | c | c | c | c | c |}
 \hline
& $G$ & $H$ & $\phi_1(G)|H$ & $K$ & $\phi_1(G)|K$ & $H\cap K$\\ \hline   

1& $\AA_{2n-1}$  & $\CC_n$ & $\phi_1$ & $\AA_{2n-2}$ & $\phi_1+\theta_1$ & $\CC_{n-1}$ \\ \hline
2 & $\DD_{n+1}$  & $\BB_n$ & $\phi_1+\theta_1$ & $\AA_n$ & $\phi_1+\phi_n$ & $\AA_{n-1}$\\ \hline
3.1 & $\DD_{2n}$  & $\BB_{2n-1}$ & $\phi_1+\theta_1$ & $\CC_n$ & $2\phi_1$ & $\CC_{n-1}$\\ \hline
3.2 & $\DD_{2n}$  & $\BB_{2n-1}$ & $\phi_1+\theta_1$ & $\CC_n\times\AA_1$ & $\phi_1\otimes\phi_1$ & $\CC_{n-1}\times\AA_1$\\ 
\hline
4.1 & $\BB_3$ & $\GG_2$ & $\phi_2$ & $\BB_2$ & $\phi_1+\theta_2$ & $\AA_1$ \\ \hline
4.2& $\BB_3$ & $\GG_2$ & $\phi_2$ & $\DD_3$ & $\phi_1+\theta_1$ & $\AA_2$ \\ \hline
5.1 & $\DD_4$ & $\BB_3$ & $\phi_3$ & $\BB_2$ & $\phi_1+\theta_3$ & $\AA_1$\\ \hline
5.2 & $\DD_4$ & $\BB_3$ & $\phi_3$ & $\BB_2\times\AA_1$ & $\phi_1+\phi_1^2$ & $\AA_1\times\AA_1$\\ 
\hline
5.3 & $\DD_4$ & $\BB_3$ & $\phi_3$ & $\DD_3$ & $\phi_1+\theta_2$ & $\AA_2$\\ \hline
5.4 & $\DD_4$ & $\BB_3$ & $\phi_3$ & $\BB_3$ & $\phi_1+\theta_1$ & $\GG_2$\\ \hline
6 & $\DD_8$ & $\BB_7$ & $\phi_1+\theta_1$ & $\BB_4$ & $\phi_4$ & $\BB_3$\\ \hline
 \end{tabular}
\end{center}
\end{table}

In our tables, we always have $n>1$ and $k\geq 1$. We use $\theta_k$ to denote a trivial representation of dimension $k$. Corresponding to an ordering of the simple roots of  $G$ we have fundamental representations $\phi_i=\phi_i(G)$, $i=1,\dots,\rank G$. We use the ordering of the roots of the   simple groups   of Dynkin \cite{Dynmax}.  Note that entries (5.1), (5.2) and (5.3) of Table 1 are special cases of (3.1), (3.2) and (2), but we have included them for completeness.

 \begin{corollary}\label{cor:rank1}
Let $(G,H,K)$ be a triple in Table 1. 
\begin{enumerate}
\item  If  $L\subset G$ is a reductive subgroup commuting with $H$ or $K$, then $L$ has rank at most 1.
\item We  have $G=H_sK_s$ where $K_s$ and $H_s$ are simple.
\end{enumerate}
\end{corollary}

Now that we know the possibilities for $G$, $H$ and $K$, our task is  to find the  irreducible representations of $G$ which remain irreducible when restricted to $H$. This can be read off from   \cite[Table 5]{Dynmax}. However, given that one knows the possibilities for $(G,H,K)$, it is relatively easy to see which irreducible representations of $G$ are possible.  Note that we can sometimes gain an irreducible representation by adding a group of rank 1 to $H$ (Table 2(3.5)).

\begin{theorem}
Let $G=G_s$ be simple and let $H$ and $K=K_s$ be proper semisimple subgroups of $G$ such that   $G=HK$. Assume that $V$ is an irreducible representation of $G$ which is also irreducible when restricted to $H$. Then, up to automorphisms of $G$, all possibilities are listed in Table 2.
\end{theorem}
 
 \begin{table}
\caption{}
\begin{center}
\begin{tabular}{| c | c | c | c | c | c | c | c |} \hline
& $G$ & $V$ & $H$ & $V|H$ &  $K$ & $V|K$ & $V^K$ \\ \hline 
1 & $\AA_{2n-1}$ & $\phi_1^k$ & $\CC_n$ & $\phi_1^k$ & $\AA_{2n-2}$ & $\phi_1^k+\phi_1^{k-1}+\dots+\theta_1$ & $\theta_1$ \\ \hline
2.1 & $\DD_{2n+1}$  & $\phi_{2n}^k$  & $\BB_{2n}$ & $\phi_{2n}^k$ & $\AA_{2n}$ & $\ss^k(\phi_1 +\phi_3 +\dots+\phi_{2n-1} +\theta_1)$ & $\theta_1$ \\ \hline
2.2 & $\DD_{2n+1}$  & $\phi_{2n+1}^k$  & $\BB_{2n}$ & $\phi_{2n}^k$ & $\AA_{2n}$ & $\ss^k(\phi_2 +\phi_4 +\dots+\phi_{2n} +\theta_1)$ & $\theta_1$\\ \hline
3.1 & $\DD_{2n}$  & $\phi_{2n-1}^k $  & $\BB_{2n-1}$ & $\phi_{2n-1}^k$ & $\AA_{2n-1}$ & $\ss^k(\phi_2 +\phi_4 +\dots+\phi_{2n-2} +\theta_2)$ & $S^k(\C^2)$\\ \hline
3.2 & $\DD_{2n}$ & $\phi_{2n-1}^k $  & $\BB_{2n-1}$ & $\phi_{2n-1}^k$ & $\CC_n$ & $*$ & $S^k(\C^{n+1})$\\ \hline
3.3 & $\DD_{2n}$  & $\phi_{2n}^k $  & $\BB_{2n-1}$ & $\phi_{2n-1}^k$ & $\AA_{2n-1}$ & $\ss^k(\phi_1 +\phi_3 +\dots+\phi_{2n-1})$ & $(0)$\\ \hline
3.4 & $\DD_{2n}$  & $\phi_{2n}^k $  & $\BB_{2n-1}$ & $\phi_{2n-1}^k$ & $\CC_n$ & $*$ & $(0)$\\ \hline
3.5 & $\DD_{2n}$  & $\phi_1 $  & $\CC_n\times\AA_1$ & $\phi_1\otimes\phi_1$ & $\BB_{2n-1}$ & $\phi_1+\theta_1$ & $\theta_1$\\ \hline
4.1 & $\BB_3$ & $\phi_1^k$ & $\GG_2$ & $\phi_2^k$ & $\BB_2$ & $\ss^k(\phi_1+\theta_2)$& $S^k(\C^2)$\\ \hline
4.2 & $\BB_3$ & $\phi_1^k$ & $\GG_2$ & $\phi_2^k$ & $\DD_3$ & $\ss^k(\phi_1+\theta_1)$& $\theta_1$\\ \hline
5.1 & $\DD_4$ & $\phi_1^k$ & $\BB_3$ & $\phi_3^k$ & $\BB_2$ & $\ss^k(\phi_1+\theta_3)$ & $S^k(\C^3)$\\ \hline
5.2 & $\DD_4$ & $\phi_1^k$ & $\BB_3$ & $\phi_3^k$ & $\DD_3$ & $\ss^k(\phi_1+\theta_2)$ & $S^k(\C^2)$\\ \hline
5.3 & $\DD_4$ & $\phi_1^k$ & $\BB_3$ & $\phi_3^k$ & $\BB_3$ & $\ss^k(\phi_1+\theta_1)$ & $\theta_1$\\ \hline
5.4 & $\DD_4$ & $\phi_1$ & $\CC_2\times\AA_1$ & $\phi_1\otimes\phi_1$ & $\BB_3$ & $\phi_1+\theta_1$ & $\theta_1$\\ \hline
6.1 & $\DD_8$ & $\phi_1$ & $\BB_4$ & $\phi_4$ & $\BB_7$ & $\phi_1+\theta_1$ & $\theta_1$ \\ \hline
6.2 & $\DD_8$ & $\phi_7$ & $\BB_4$ & $\phi_1\phi_4$ & $\BB_7$ & $\phi_7$  &$(0)$ \\ \hline
6.3 & $\DD_8$ & $\phi_7^k$ & $\BB_7$ & $\phi_7^k$ & $\BB_4$ & $*$ & $S(f_4)_k$\\ \hline
6.4 & $\DD_8$ & $\phi_8^k$ & $\BB_7$ & $\phi_7^k$ & $\BB_4$ & $*$  &$S(f_2,f_3)_k$\\ \hline

\end{tabular}
\end{center}
\end{table}

In Table 2, if $V$ is a $K$-module, then $\ss^k(V)$ denotes the $K$-subspace of $S^k(V)$ generated by $S^k(V^U)$ where $U$ is a maximal unipotent subgroup of $K$. In other words, in $\ss^k(V)$ we take only the Cartan components of products. In column $V^K$ the notation $S(f_2,f_3)_k$ means the span of the monomials in $f_2$ and $f_3$ of degree $k$ where $f_i$ has degree $i$, $i=2$, $3$. Here the $f_i\in\C[\phi_8(D_8)]^{B_4}$. A similar  interpretation applies to  $S(f_4)_k$ where $f_4\in \C[\phi_7(D_8)]^{B_4}$ has degree $4$. The justification of the entries   $V|K$ and $V^K$ can be found in the Appendix. 

\begin{remark} In Table 2, we have chosen not to remove all redundancies  due to automorphisms of groups of type $D_n$.
In column $V|K$ we have omitted the decompositions in (3.2) and (3.4) which are obtained by restricting $V$ to $\CC_n$. To determine this one needs the branching rule for restrictions of $\SL_{2n}$-representations to $\CC_n$. As determined by Weyl \cite{Weyl}, one proceeds as follows. Let $\omega\in\bigwedge^2((\C^{2n})^*)$ be nonzero and $\CC_n$-invariant. Then from the exterior powers of $\omega$ we obtain invariants in the duals of $\phi_2^{b_2}\dots\phi_{2n-2}^{b_{2n-2}}$ for nonnegative $b_j$. Then the restriction of an irreducible representation $\phi$ of $\SL_{2n}$ to $\CC_n$ is obtained by taking all possible complete  contractions of the duals of our invariants in the $\phi_2^{b_2}\dots\phi_{2n-2}^{b_{2n-2}}$ with $\phi$. For example, $\phi_1\phi_5(\SL_8)$ contracted with $\omega$ gives rise to $\phi_4$ and $\phi_1\phi_3$ while contraction with $\omega\wedge\omega\in\wedge^4(\C^{8})^*$  gives rise to $\phi_1^2$ and $\phi_2$. Note that the only representations of $\SL_{2n}$ which can give rise to the trivial representation of $\CC_n$ are those of the form $\phi_2^{a_2}\dots\phi_{2n-2}^{a_{2n-2}}$. Hence in Table 2 (3.4) the trivial $\CC_n$-representation does not occur in the column $V|K$ for any $k\geq 1$ and in (3.2) the occurrences of the trivial representation are the symmetric algebra in $\theta_2$ and the subspaces $\phi_2^{\CC_n},\dots,\phi_{2n-2}^{\CC_n}$.
\end{remark}
\begin{remark}
We do not know what to put in the column $V|K$ in cases (6.3) and (6.4) of Table 2. However, in the Appendix we are able to compute $V^K$. 
\end{remark}

Suppose that $H\subset G$ where $G$ and $H$ are  semisimple, and  connected  and   $V$ is a $G$-module which is irreducible as an $H$-module. Then the inclusion of $H$ in $G$ has a very special form, as shown by Dynkin \cite[Theorem 2.2]{Dynmax}.

\begin{theorem}\label{thm:irred}
Let $G_1,\dots,G_k$ be the simple components of  $G$. Then $H=H_1\cdots  H_k$ where the $H_i$ are nontrivial semisimple subgroups of the $G_i$, $i=1,\dots, k$.
\end{theorem}
 
 We are interested in the case that $H$ is simple. Then Theorem \ref{thm:irred} tells us that $G$ has to be simple and from Table 2 we get the following theorem.

\begin{theorem} \label{thm:simple}Let $H$ be simple of rank at least two and let $\phi\colon H\to\GL(V)$ be an irreducible representation. Then every  nonzero $H$-orbit $Hv$  is semi-characteristic, except for the following cases (where   $n\geq 2$ and $k\geq 1$).
\begin{enumerate}
\item $H=\CC_n$,  $\phi=\phi_1^k$ and $v$ is a highest weight vector. Equivalently, $v$ is fixed by $\AA_{2n-2}$ where $\AA_{2n-2}$, $\CC_n$,  $\AA_{2n-1}$ and $V$ are as in Tables  1(1) and 2(1).
\item $H=\BB_{2n}$,   $\phi=\phi_{2n}^k$ and $v$ is a highest weight vector. Equivalently, $v$ is fixed by $\AA_{2n}$ where $\AA_{2n}$, $\BB_{2n}$, $\DD_{2n+1}$ and $V$ are as in Tables   1(2) and  2(2.1) or 2(2.2).
\item $H=\BB_{2n-1}$,  $\phi=\phi_{2n-1}^k$ and $v$ is fixed by $\CC_n$ where $\CC_n$, $\BB_{2n-1}$, $\DD_{2n}$ and $V$ are as in Tables  1(3.1) and   2(3.2).
\item $H=\GG_2$, $\phi=\phi_2^k$ and $v$ is fixed by $\BB_2$ where $\BB_2$, $\GG_2$, $\BB_3$ and $V$ are as in Tables  1(4.1) and   2(4.1).
\item $H=\BB_4$, $\phi=\phi_4$ and $H  v$ is closed. Equivalently, $v$ is fixed by $\BB_7$ where $\BB_4$, $\BB_7$, $\DD_8$ and $V$ are as in   Tables 1(6) and  2(6.1).
\item $H=\BB_7$, $\phi=\phi_7^k$ and $v$ is fixed by $\BB_4$ where $\BB_4$, $\BB_7$, $\DD_8$ and $V$ are as in Tables  1(6) and 2(6.3) or 2(6.4).   
\end{enumerate}
\end{theorem}

For special direct sums of representations we have the following result.

\begin{proposition}
Let $V_i$ be an irreducible $H_i$-module where the $H_i$ are semisimple, $i=1,\dots,k$. Let $V=V_1\oplus\dots\oplus V_k$ be the corresponding $H=H_1\times\dots\times H_k$-module. Suppose that $0\neq v_i\in V_i$ such that $v_i$ is (semi)-characteristic for $H_i$ for each $i$. Then $v$ is (semi)-characteristic for $H$ where $v=(v_1,\dots,v_k)\in V$.
\end{proposition}

\begin{proof} 
Let $G$ be as usual. First suppose that $V$ is an irreducible $G^0$-module. Then, up to a cover and scalar matrices,  $G^0=G_1\times\dots\times G_r$ where the $G_j$ are simple and $V\simeq U_1\otimes \dots\otimes U_r$ where the $U_j$ are irreducible $G_j$-modules. By Theorem \ref{thm:irred} each simple factor of each $H_i$ must project nontrivially to a single $G_j$. But given the structure of $V$ as $H$-module, this implies that $k=1$, where the theorem is trivial.

We may now assume that there is a maximal flag $W_1\subset\dots\subset W_r\subset W_{r+1}=V$ of $G^0$-stable subspaces where $r\geq 1$. We may assume that, as $H$-module,  $W_r=V_1\oplus\dots\oplus V_p$ so that $V/W_r\simeq V_{p+1}\oplus\dots\oplus V_k$. The image of $G$ in $\GL(V/W_r)$ is reductive. Let  $G'$ denote its semisimple part. Then for $g\in G'$, we have $g(v_{p+1},\dots,v_k)\in (H_{p+1}\times\dots\times H_k)(v_{p+1},\dots,v_k)$. By induction on $k$, $G'=H_{p+1}\times\dots\times H_k$. But by maximality of the flag, we must have that $p=k-1$, i.e., $V/W_r\simeq V_k$. If $V_k$ is not $G$-stable, then  $\lieg$ contains a nonzero linear map of $V_k$ to $W_r$. Since $\lieg$ is stable under the action of $H$, we may assume that it contains $\Hom(V_k,V_1)$. Thus the $G$-orbit of $v$ contains a point $(0,v_2,\dots,v_k)$. Such a point is not in $Hv$, so we have a contradiction. Thus $V_k$ is 
 $G$-stable and we have a $G$-module direct sum decomposition  $V=W_r\oplus V_k$. It follows by induction on $k$ that $Hv$ is (semi)-characteristic.
\end{proof}

If one considers the adjoint representation $\lieh$ of a simple $H$, the only case  that appears  in Theorem \ref{thm:simple} is $H=\CC_n$, $n\geq 2$, where $\lieh=\phi_1^2$. Thus we have

\begin{corollary}
Let $H=H_1\times\dots\times H_k$ where the $H_i$ are simple,  and let $\phi\colon H\to\GL(\lieh)$ be the adjoint representation. Let $v=(v_1,\dots,v_k)\in \oplus_i\lieh_i$ where no $v_i$ is zero. Then $v$  is semi-characteristic if and only if for every simple component $H_i$ of type $\CC_n$, $n\geq 2$, $v_i\in \lie c_n$ is not on the   highest weight orbit.
\end{corollary}

\section{The case $H$ is simple}
We now consider the case where $H$ is simple of rank at least two and our $H$-module $V$ may be reducible. We consider the possible semisimple $G$ which can act on $V$ such that $Gv=Hv$ where $v\in V$ is  generic for the action of $H$.
  
 Here are some examples to keep in mind.

 \begin{example}\label{ex:usualdouble}
 Let $H$ be simple and let $V$ be an $H$-module. Let $G=H\times H$ and let $v$ be the identity in $V\otimes V^*$.  Then $Gv=Hv$ where the diagonal copy of $H$ in $G$ plays the role of $K$.
 \end{example}
 
 \begin{example}\label{ex:usualK}
 Let $(G,H,K)$ or $(G,K,H)$  be an entry of Table 1. Let $V=\bigoplus_{i=1}^n m_iV_i$ be the isotypic decomposition of a $G$-module such that  $\dim V_i^K\geq m_i\geq 1$ for all $i$. Let $v\in V^K$ be generic for the action of $G$. Then   $Gv=Hv$. 
 \end{example}

 \begin{example} \label{ex:a2n-1} Here $H\simeq \AA_{2n-1}$. Let $(G_1,H_1,K_1)=(\DD_{2n},\AA_{2n-1},\BB_{2n-1})$  and  $(G_2,H_2,K_2)=(\AA_{2n-1},\AA_{2n-2},\CC_n)$. 
 Let $G=G_1\times G_2$, let $H$ denote the diagonal copy of $\AA_{2n-1}$ and let $K=K_1\times K_2$. Then $G=HK$. Let $V$ be a representation of $G$ which contains a generic vector $v\in V^K$.  If $V$ is irreducible, then $V=V_1\otimes V_2$ where $V_i$ is an irreducible representation of $G_i$, $i=1$, $2$, and  $v\in V_1^{K_1}\otimes V_2^{K_2}$. The only possibilities allowing nontrivial fixed points are $V_1=\phi_1^k$, $k\geq 0$, and $V_2=\phi_2^{a_2}\phi_4^{a_4}\cdots\phi_{2n-2}^{a_{2n-2}}$ where the $a_{2i}$ are in $\Z^+$. In both cases, $\dim V_i^{K_i}=1$. Thus for $v$ to be generic,  $V$ must be a sum of representations (each of multiplicity one) of the form $V_1\otimes V_2$  where $\dim V_i^{K_i}=1$ for all $i$. For $V$ to be almost faithful the sum must contain a nontrivial $V_1$  and a nontrivial $V_2$.
 \end{example}
  
  \begin{example} \label{ex:a3} Here $H\simeq \AA_3$. 
  Let $(G_1,H_1,K_1)=(\BB_3,\DD_3,\GG_2)$ and $(G_2,H_2,K_2)=(\AA_3,\AA_2,\CC_2)$. Let $G=G_1\times G_2$, let   $H$ be the diagonal copy of $\AA_3=\DD_3$ and let $K=K_1\times K_2$. Then $G=HK$. If $V_i$ is an irreducible representation of $G_i$, $i=1$, $2$,   with $V_1^{K_1}\neq (0)\neq V_2^{K_2}$, then $V_1=\phi_3^k$, $k\geq 0$ and $V_2=\phi_2^\ell$, $\ell\geq 0$. Again, $\dim V_i^{K_i}=1$, $i=1$, $2$, and the conditions for $v$ generic and $V$ almost faithful   are as in the case above.
    \end{example}

  \begin{example} \label{ex:b3} Here $H\simeq\BB_3$. Let  $(G_1,H_1,K_1)=(\DD_4,\BB_3,\BB_3)$ and $(G_2,H_2,K_2)=(\BB_3, \GG_2, \BB_2$ or $\DD_3)$. Let $G=G_1\times G_2$, let $H\subset H_1\times G_2$  be the diagonal copy of $\BB_3$ and let $K=K_1\times K_2$. Then $G=HK$. The possibilities for the $V_i$   having nontrivial $K_i$-fixed points are $V_1=\phi_1^k$, $k\geq 0$ and $V_2=\phi_1^a\phi_2^b$ if $K_2=\BB_2$ and $V_2=\phi_1^a$ if $K_2=\DD_3$ where $a$ and $b$ are nonnegative. While $V_1^{K_1}$ has dimension 1, this is not true for $V_2^{K_2}$, in general, if $K_2=\BB_2$.  Let $v\in V$ be generic and $K$-fixed. Then    irreducible $G$-modules which can occur in $V$ are sums of  tensor products of  modules $V_1\otimes V_2$ where $\dim V_i^{K_i}\geq 1$, $i=1$, $2$.  
  \end{example}

  \begin{example}\label{ex:double} Let $(G_1,H_1,K_1)$ be an entry of Table 1. Let $G=G_1\times G_1$, let $H$ be the diagonal copy of $G_1$ and let $K=H_1\times K_1\subset G$. Then $G=HK$. It is usually easy to determine the almost faithful $V_1\otimes V_2$ with $K$-fixed points.  For example, in the case of  $(\AA_{2n-1},\CC_n,\AA_{2n-2})$, $V_1$ has to be of the form $\phi_2^{a_2}\cdots\phi_{2n-2}^{a_{2n-2}}$ where the $a_{2i}$ are nonnegative and $V_2$ has to be of the form $\phi_1^k$ or $\phi_{2n-1}^k$ for $k\geq 0$.   On the other hand, if we have $(\DD_8,\BB_7,\BB_4)$, then $V_1$ is of the form $\phi_1^k$, $k\geq 0$, and we have been unable to pin down exactly which $V_2$ have $\BB_4$ fixed points.   
   \end{example}
  
  Looking at Table 1 one easily sees the following.
  
  \begin{proposition}\label{prop:table1}
  Suppose that  $(G,H,K)$ and $(G',H',K')$ appear  in Table 1 where $H\cap K\simeq H'\subset L\subset G'$ or $H\cap K\simeq K'\subset L\subset G'$ and $L$ is isomorphic to $H$ or $K$.  Then $G'$ is isomorphic to $L$.
  \end{proposition}
 
 \subsection{}\label{subsec:setup}
  Let $H$ be simple  of rank at least $2$ and let $V$ be an almost faithful $H$-module. Let $v\in V$ be generic. Let $G=G_1\times\cdots\times G_r$ where the $G_i$ are simple and simply connected and $G$ acts almost faithfully on $V$ such that  $Gv=Hv$ where $H\subset G$. Let $K$ denote a Levi factor of $G_v$. Then $G=HK$. Let $\pr_i\colon G\to G_i$ denote projection on the $i$th factor, $i=1,\dots,r$. We may assume that $\pr_i(H)\neq\{e\}$ for $1\leq i\leq s$ and $\pr_i(H)=\{e\}$ for $s<i\leq r$ where $s\geq 1$. Let $G'=G_1\times\dots\times G_s$ and $G''=G_{s+1}\times\cdots\times G_r$. For $r\geq j>s$ there is a unique simple component $K_j$ of $K$ such that $\pr_j(K_j)=G_j$   and clearly $K'':=K_{s+1}\times\cdots\times K_{r}$ covers $G''$. The kernel of $K''\to G'$ commutes with $H$ and fixes $v$. Since $Hv$ spans $v$ (Proposition \ref{prop:generic}), the kernel  must be finite. Hence  $K''$ covers its image in $G'$. Let   $K'$ denote the product of the simple components of $K$ not in $K''$. Then $K'\subset G'$ and the projection of $K''$ to $G'$ centralizes $K'$. We must have that $HK'=G'$. 
  
  We may write $V=\bigoplus V_i\otimes W_i$ where the $V_i$ are pairwise nonisomorphic irreducible representations of $G'$ and the $W_i$ are representations of $G''$. Then the projection $v_i$ of $v$ to each $V_i\otimes W_i$ is a tensor of rank $\dim W_i$ since $v$ is generic. Let $U_i$ denote the smallest subspace of $V_i$ such that $v_i\in U_i\otimes W_i$. Then $\dim U_i=\dim W_i$ and $v_i\in U_i\otimes W_i$ corresponds to a $K''$-equivariant isomorphism of $W_i^*$ onto $U_i\subset V_i$. In the sense of the following definition, $v_i$ corresponds to a \emph{subordination\/} $\alpha_i\colon (W_i^*,G'') \to (V_i,G')$.

    \begin{definition}
Let $Z_i$ be an $L_i$-module $i=1$, $2$, where the $L_i$ are reductive.   We say that $Z_1$ is \emph{subordinate \/} to $Z_2$ if there is a  linear injection $\alpha\colon Z_1\to Z_2$ and a reductive subgroup $L\subset L_1\times L_2$  such that $\alpha$ is $L$-equivariant (for the $L$-module structures on $Z_1$ and $Z_2$). Moreover, we require that $\pr_1\colon L\to L_1$ be a cover.  
We say that $\alpha\colon Z_1\to Z_2$ is a \emph{subordination\/} of $Z_1$ to $Z_2$. We sometimes use the notation $\alpha\colon (Z_1,L_1)\to (Z_2,L_2)$ to specify the groups involved.
\end{definition}

 We now consider the possibilities for $K'$.
 
  \begin{lemma}\label{lem:pi_i(K)}
  Let $H$, etc. be as in \eqref{subsec:setup}. Then for $1\leq i\leq s$ we have $\pr_i(K')\neq G_i$.
  \end{lemma}
  
   \begin{proof}
  Suppose that $\pr_i(K')=G_i$. Then there is a unique simple component $K_i$ of $K'$ such that $\pr_i(K_i)=G_i$. If $\pr_j(K_i)=\{e\}$ for $j\neq i$, then $G_i$ acts trivially on $Gv$, which is not possible. Hence  $\pr_j(K_i)\neq \{e\}$ for some $j\neq i$,  $1\leq j\leq s$.  Suppose that $\pr_j(K_i)=G_j$. Then no simple component of $K'$ other than $K_i$ can project nontrivially to $G_i$ and $G_j$. Consider the projections $H'$  of $H$ and $K'_i$ of $K_i$ to $G_i\times G_j$. Then $H'K'_i=G_i\times G_j$, and by reason of dimension we must have that $\pr_i(H')=G_i$ and $\pr_j(H')=G_j$. On the level of Lie algebras this says that we have a simple Lie algebra $\lieg$ and two subalgebras $\lieh_1$ and $\lieh_2$ of $\lieg\oplus\lieg$ which project isomorphically to each $\lieg$  factor such that $\lieh_1+\lieh_2=\lieg\oplus\lieg$. But it follows from  \cite[Theorem 9]{Jacobson} that $\lieh_1\cap\lieh_2\neq (0)$, a contradiction.
  
  We may thus assume that $\dim\pr_j(K_i)<\dim G_j$, hence $\dim H<\dim G_j$ as well. By Corollary \ref{cor:rank1} we   have that 
$\pr_j(K_i)\pr_j(H)=G_j$ and that $\pr_j(K')$  differs from $\pr_j(K_i)$ by at most a factor of rank 1.  Hence, with $H'$ and $K_i'$ as above, we again have $H'K_i'=G_i\times G_j$ which is not possible by reason of  dimension. Hence we have that $\pr_i(K')\neq G_i$ for $1\leq i\leq s$.
  \end{proof}

  \begin{corollary}\label{cor:pi_j(K_i)}
 Suppose that $1\leq i\leq s$ and that $K_i\subset K'$ is a simple factor of rank at least two such that $\pr_i(K_i)\neq \{e\}$. Then $\pr_j(K_i)=\{e\}$ for $i\neq j$, $1\leq j\leq s$.
  \end{corollary}
  
  \begin{proof}
  Suppose that $\pr_j(K_i)\neq \{e\}$.  If $\dim H<\dim G_i$ and $\dim H<\dim G_j$, then Corollary \ref{cor:rank1} shows  that $\pr_i(K_i)$ and $\pr_j(K_i)$ differ from $\pr_i(K')$ and $\pr_j(K')$ by at most groups of rank 1, so that with $H'$ and $K_i'$ as above, we have $H'K_i'=G_i\times G_j$, which is not possible by reason of dimension. Thus $\pr_i(H)=G_i$ (or $\pr_j(H)=G_j$) which forces $\pr_j(K_i)=G_j$ (or $\pr_i(K_i)=G_i$), contradicting the lemma above.
    \end{proof}
    
    \begin{theorem}\label{thm:Hsimple}
    Let $H$, $s$, $r$, $V=\bigoplus_i V_i\otimes W_i$, etc.\ be as above. Then one of the following occurs.
    \begin{enumerate}
\item $s=1$ and $H=G_1$. Then $G''=G_2\times\cdots\times  G_r$ and $v$ corresponds to   a subordination $(\bigoplus_i W_i^*,G'')\to(\bigoplus_i V_i,H)$ where $W_i^*$ is sent to $V_i^{K'}$ for each $i$.  
\item $s=1$, $H\neq G_1$ and $G_1=HK_1$ where $K_1\subset G_1$ is a simple component of $K$. If $r>1$, then  $r=2$, $G_2=\SL_2$ (up to a cover) and $(G_1,H,K_1\times\SL_2)$ is case 3.2 of Table 1. We have  a subordination    $(\bigoplus W_i^*,\SL_2)\to (\bigoplus V_i,H)$ where the image of $W_i^*$ is a subset of $V_i^{K_1}$ for each $i$.
\item $s=2$, $\pr_1(H)=G_1$, $\pr_2(H)=G_2$ and there are  subgroups $K_1'$, $K_2'$ of $H$ such that $(H,K_1',K_2')$ occurs in Table 1. We have  $K=K_1\times K_2$ where $K_i=\pr_i(K_i')$, $i=1$, $2$  and $G=HK$ (Example \ref{ex:double}). If $r>2$, then $r=3$, $(H,K_1',K_2'\times\SL_2)$ or $(H,K_2',K_1'\times\SL_2)$ is entry 3.2 of Table 1, $G''=\SL_2$ and $v$ corresponds to a subordination  $( \bigoplus_i W_i^*,\SL_2) \to(\bigoplus_i V_i,H)$ where $W_i^*$ has image in $V_i^{K_1\times K_2}$ for all $i$.
\item $r=s=2$ where $\pr_1(H)\neq G_1$ and $\pr_2(H)=G_2$. Then we are in the case  of Example  \ref{ex:a2n-1}, \ref{ex:a3} or \ref{ex:b3}. 
\end{enumerate}

 \end{theorem}
 
 \begin{proof}
 The cases where $s=1$ are quite easy and we leave them to the reader. Now suppose that $s>1$.  Suppose that $\pr_1(H)\neq G_1$ and that $\pr_2(H)\neq G_2$. Then there are strongly semisimple factors $K_i$ of $K'$ such that $\pr_i(K_i)\pr_i(H)=G_i$, $i=1$, $2$. By Lemma \ref{lem:pi_i(K)} and Corollary \ref{cor:pi_j(K_i)} the $\pr_i(K_i)$ are proper subgroups of the $G_i$ where $\pr_2(K_1)=\pr_1(K_2)=\{e\}$. Applying Proposition \ref{prop:table1} we obtain that each of the $G_i$ is isomorphic to $H$, a contradiction.  Thus we may assume that $\pr_1(H)=G_1$.
 
 Let $L'$ (resp.\ $K_1$) be the product of the strongly simple components of $K'$ which map trivially (resp.\ nontrivially) to $G_1$. Then $L'\subset G_2\times\dots\times G_s$. Since $G'=HK'$ and $\pr_1(H)=G_1$, we must have that $G_2\times\dots\times G_s=H'L'$ where $H'$ is the  inverse image of $\pr_1(K_1)$ in $H$ projected  to $G_2\times\dots\times G_s$.  Since $H'$ is a proper subgroup of $H$, $\pr_i(H')\neq G_i$, $i=2,\dots,s$. Since $\pr_2(H')\pr_2(L')=G_2$, we must have that $H'$ is simple, by Table 1. By our argument above, we must have that $s=2$.

 Now suppose that $\pr_1(H)=G_1$ and that $\pr_2(H)=G_2$. Then by Lemma \ref{lem:pi_i(K)} and Corollary \ref{cor:pi_j(K_i)} we have that $H\times H\simeq G_1\times G_2 =H(K_1\times K_2)$ where the $K_i\subset G_i$ are images of simple subgroups $K_1'$ and $K_2'$ of $K$. Then $H=K_1'K_2'$ so that $(H,K_1',K_2')$ occurs in Table 1, as claimed. Suppose that $r>2$. Then for $j>2$, $K_j\subset G_1\times G_2\times G_j$ projects onto $G_j$ and commutes with $K_1$ and $K_2$. But the centralizer of $K_1\times K_2$ in $G_1\times G_2$ is trivial unless $(H,K_1',K_2')$ or $(H,K_2',K_1')$ is entry 3.1 of Table 1, in which case the centralizer is $\SL_2$. Thus $r=3$ and there is a subordination as claimed.
 
 Finally, suppose that $s=2$ and that $\pr_2(H)=G_2$ and $\pr_1(H)\neq G_1$. Using Lemma \ref{lem:pi_i(K)} and Corollary \ref{cor:pi_j(K_i)} and the fact that $HK'=G_1G_2$, we see that there are simple subgroups $K_i\subset G_i$, $i=1$, $2$, such that $H(K_1 K_2)=G_1G_2$. We have entries $(G_1,\pr_1(H),K_1)$ and $(\pr_2(H),\pr_2(L),K_2)$ in Table 1 where $L$ is the preimage in $H$ of $\pr_1(H)\cap K_1$. Then $H$ must be   $\BB_3$ or of type $\AA_{2n-1}$. If $H=\BB_3$, then one easily sees that we are  in Example \ref{ex:b3} and that we cannot have $r>2$. The remaining possibilities are that  $H=\AA_{2n-1}$ and $G_1=\DD_{2n}$ or $\BB_3$ giving Examples \ref{ex:a2n-1} and \ref{ex:a3} where $r=2$ is forced.
  \end{proof}

The theorem above gives one the possibilities for the semisimple part of the Levi factor of $\{g\in\GL(V)\mid Gv=Hv\}$.   Preferable would be a  theorem which starts with a representation $V$ of $H$ and a generic $v\in V$ and tells you when $v$ is almost semi-characteristic.  In general, it is rather cumbersome to give such a theorem (for $\SL_2$ see section \ref{sec:sl2}). We content ourselves with working out the following example.

\begin{example}\label{ex:d2n+1} 
Let $H=\DD_{2n+1}$, $n\geq 2$. Let $V=\bigoplus_{i=1}^k n_iV_i$ be the isotypic decomposition of the $H$-module $V$. Let $v=(v_1,\dots,v_k)\in V$ be generic. We find conditions which guarantee that $v$ is almost semi-characteristic. 

Each $v_i$ is $(v_{i1},\dots,v_{i,n_i})$ where $v_{ij}$ lies in the $j$th copy of $V_i$, and the $v_{ij}$ span  a subspace $U_i\subset V_i$ of dimension $n_i$. In order to avoid case (1) of Theorem \ref{thm:Hsimple} we have to assume that the intersection of the  stabilizers of the subspaces $U_i$ in $H$ contains no nontrivial semisimple group. Cases (2) and (4) do not apply, so we are left with case (3), where we have $G=H\times H$, $K_1=\BB_{2n}$ and $K_2=\AA_{2n}$. But then there is a copy of $\AA_{2n-1}$ in $\DD_{2n+1}$ which fixes our point. We have already ruled this out.
\end{example}
 
\section{Semisimple groups}

We turn our attention to the case that $H\subset G$ where $G$ and $H$ are  connected semisimple, $V$ is an irreducible  $H$-module, $G$ acts almost faithfully on $V$ and $Gv=Hv$ for some nonzero $v\in V$. Let  $G_1,\dots,G_k$ be the  simple components of $G$. Then Theorem \ref{thm:irred}  tells us that $H=H_1\cdots H_k$ where the $H_i$ are semisimple and lie in $G_i$, $i=1,\dots,k$. Note that no $H_i$ is trivial, else $G_i$ acts trivially on $V$. Thus if $G_i$ has rank 1, then $G_i=H_i$. We have $V=V_1\otimes\dots\otimes V_k$ where $V_i$ is an almost faithful irreducible representation of both $G_i$ and $H_i$, $i=1,\dots,k$.

\subsection{}\label{subs:I}
Suppose that $G=HK$ where $K$ is semisimple and $G$, $H$ and $V$ are as above. (Think of $K\subset G_v$.)\ Let $\pr_i$ denote the projection of $G$ to $G_i$, $i=1,\dots,k$. Let $K'$ be a simple component of  $K$ and set $I':=\{i\mid H_i\pr_i(K')=G_i$ and $H_i\neq G_i\}$.  We may assume that $K$ contains no simple component of rank 1. 
\begin{proposition}\label{prop:I}
Let $G=HK$ as above. Let  $K'$, $K''$ be distinct simple components of $K$ and let $I'$ and $I''$ be as above. Then 
\begin{enumerate}
\item  $I'\cap I''=\emptyset$.
\item $|I'|\leq 2$. If $|I'|=2$, then $\pr_i(K')=G_i$ for some $i\in I'$.
\end{enumerate}
\end{proposition}

\begin{proof} For any $i\in I'\cap I''$, the (nontrivial) images of $K'$ and $K''$ in $G_i$ commute. This is clearly not possible if $\pr_i(K')$  is $G_i$. If not, then we are in one of the entries of Table 1, and commutativity is not possible if $\pr_i(K')$ is one of the groups occurring there. Hence (1) holds.
Suppose that $i$, $j\in I'$, $i\neq j$ and $\pr_i(K')\neq G_i$ and $\pr_j(K')\neq  G_j$. Then we must have that $H_jL=G_j$ where $L=\pr_j(\pr_i\inv(H_i)\cap K')$ is a proper subgroup of $\pr_j(K')$ as in the last column of Table 1. But then, by inspection,  we cannot have $H_jL=G_j$. If $i$, $j$ and $k$ are distinct elements of $I'$, then we can assume that $\pr_i(K')=G_i$, and we derive a contradiction as before by considering the non-surjective projections of $\pr_i\inv(H_i)\cap K'$ to $G_j$ and $G_k$. Thus we have (2).
\end{proof}

\begin{theorem}\label{thm:k=2}
Suppose that $k=2$ and that  $H_1\neq G_1$, $H_2\neq G_2$ and $Gv=Hv$ for a nonzero $v\in V$. Then we are in one of the following cases.
\begin{enumerate}
\item There are tuples $(G_i,V_i,H_i,K_i)$  in Table 2, $i=1$, $2$, and $v\in V_1^{K_1}\otimes V_2^{K_2}$.
\item The tuple  $(G_1,V_1,H_1,K_1)$ is entry (1) of Table 2,   $(G_2,V_2,H_2,K_2)$ is entry (3.3) (with the same $k$) and $v$ generates the one-dimensional space of $\AA_{2n-1}$ fixed vectors in $V_1\otimes(V_2|\AA_{2n-1})$.
\item The tuple $(G_1,V_1,H_1,K_1)$ is entry (6.2) of Table 2,   $(G_2,V_2,H_2,K_2)$ is entry (6.3) (with $k=1$) and $v$ generates the one-dimensional space of $\DD_8$ fixed points in $V_1\otimes V_2$.
\end{enumerate}
\end{theorem}

\begin{proof} Let $K_1$ denote a maximal  strongly semisimple subgroup of $G_v$. Suppose that  $K_1\subset G_1$. Then Table 1 implies that $H_1K_1=G_1$. Let $K_2$ be a strongly semisimple subgroup of $G_v$ such that $\pr_2(K_2)H_2=G_2$. Then we must have $\pr_1(K_2)=\{e\}$ (again by Table 1), hence $K_2\subset G_2$ and we are in case (1).
 Thus we may suppose that 
any maximal strongly semisimple subgroup $L$ of $G_v$ lies diagonally in $G_1\times G_2$.  Since we are not in case (1), $\pr_2$ restricted to $L$ is almost faithful (and so is $\pr_1$). By Table 1, $L$ must be simple. It follows from Proposition \ref{prop:I} that we have two cases:
\smallskip

\noindent Case 1: $\pr_1(L)=G_1$ and $\pr_2(L)=K_2$ where $(G_2,H_2,K_2)$ is in Table 1. Moreover, $(G_2,H_2,K_2')$ is in Table 1, where $K_2'=\pr_2(\pr_1\inv(H_1)\cap L)$. Thus we are in the case    $(G_2,H_2,K_2)=(\DD_{2n},\BB_{2n-1},\AA_{2n-1})$ and $(G_2,H_2,K_2')=(\DD_{2n},\BB_{2n-1},\CC_n)$, $n\geq 2$, where $H_1\simeq\CC_n$. Then from Table 2 we see that $V_1\simeq\phi_1^k(\AA_{2n-1})$ (or its dual). 
 From Table 2(3.3) we get possibility (2) of our theorem. From (3.1) we get nothing since $\AA_{2n-1}$ has no fixed vectors in $V_1\otimes V_2$. The possibilities (4.2) and (5.2) fail for the same reason.  Hence we only get (2).
\smallskip

\noindent Case 2: Here we have that $\pr_1(L)=G_1$ and $\pr_2(L)=G_2$. Then $L=H_1'H_2'$ where   $H_i'=\pr_i\inv(H_i)\cap L$, $i=1$, $2$.  Moreover, there are irreducible representations $V_i$  of $L$ whose restrictions to $H_i'$ are irreducible, $i=1$, $2$. Table 2 tells us that we may have possibilities from entry (5.3) (and isomorphic entries), but   then there are no $\DD_4$-fixed points in $V_1\otimes V_2$. Finally, from   (6.2) and (6.3) we get possibility (3) above.
\end{proof}

\subsection{}\label{subs:v} Let $H\subset\GL(V)$ where $H$ is semisimple connected and $V$ is an irreducible $H$-module. Suppose that $v\in V$ is a  nonzero orbit such that the connected semisimple part $G$ of $\{g\in\GL(V)\mid gHv = Hv\}$  is strictly larger than $H$.  Let $K$ denote the strongly semisimple part of $G_v$. We are then in the   situation of \ref{subs:I}. For each simple component $K_j$ of $K$, let $I_j\subset\{1,\dots,k\}$ be as in \ref{subs:I}. Set $I'=\cup_j I_j$, $V'=\otimes_{i\in I'} V_i$ and $V''=\otimes_{i\not\in I'} V_i$. Define $G'$ and $G''$ analogously. Then $G''=H''=\Pi_{i\not\in I'}H_i$.
Let  $K'$ be the product of the $K_j$ such that $K_j\subset G'$ and let $K''$ be the product of the other simple factors of $K$  so we have $K=K'K''$.  Via the projections to $G'$ and $H''$ we have $K''$-module structures on $V'$ and $V''$.   Let $W'\subset V'$    be the minimal subspace such that $v\in W'\otimes V''$. Then $W'\subset (V')^{K'}$ and $v$ is $K''$-fixed.  

\begin{remark}
It follows from Theorem \ref{thm:k=2} that each simple component of $K'$ arises from an entry of Table 2 as in Theorem \ref{thm:k=2}(1), is a group $\AA_{2n-1}$ as  in Theorem \ref{thm:k=2}(2) or is the group $\DD_8$ in Theorem \ref{thm:k=2}(3).
\end{remark}

We   restate the discussion in \eqref{subs:v} as follows.

\begin{theorem} Let $v\in V$. If $Hv$ is not semi-characteristic,  then there are $K'$, $K''$, etc. as in \eqref{subs:v} and a minimal $K''$-stable subspace  $W'\subset (V')^{K'}$ such that $v\in W'\otimes V''$. If $K''\neq\{e\}$, then there is  a subordination   $\alpha\colon ((W')^*,K'')\to (V'',H'')$.
\end{theorem}

Now we would like to find some simple sufficient criteria for all generic $v$ to be semi-characteristic. For this, we only need to avoid the case that $Gv=Hv$ where $G_i$ differs from $H_i$ for only one $i$. 
 Then after renumbering we have that $G_1\neq H_1$ and $G_i=H_i$ for $i>1$. We have that $V''=V_2\otimes\dots\otimes V_k$ and $H''=H_2\times\cdots\times  H_k$.  Note that $H_1$ may be any semisimple subgroup of $H$.

\begin{proposition}\label{prop:g1}
Let $G_1\supset H_1$ be as above. Then  one of the following occurs.
\begin{enumerate}
\item There is a subordination $\alpha\colon (V_1^*,G_1)\to (V'',H'')$ where $K''$ projects onto $G_1$.
 \item The tuple $(V_1,G_1,H_1,K_1)$ occurs in Table 2 where $K_1$ is the projection of $K''$ to $G_1$, and we have a subordination $(W_1^*,K_1)\to (V'',H'')$ where $W_1$ is minimal such that $v\in W_1\otimes V''$.
\item The group $K''$ projects trivially to $G_1$ and the tuple $(V_1,G_1,H_1,K_1)$ occurs in Table 2 for some $K_1$ where $V_1^{K_1}\neq (0)$. 
\end{enumerate}
\end{proposition}

\begin{proof} If $\pr_1(K'')=\{e\}$, then we are in case (3). Suppose that $\pr_1(K'')$ is nontrivial. Then it follows from Table 1 that $K'=\{e\}$. If the projection of $K''$ to $G_1$ is   $G_1$, then $v$ corresponds to a subordination of $V_1^*$ to $V''$ and we are in case (1). The only other possibility is that the projection of $K''$ is $K_1$ where $(V_1,G_1,H_1,K_1)$ occurs in Table 2 and we are in case (2).  
\end{proof}

\begin{example}
Suppose that $k=2$,  $\dim V_2\geq \dim V_1$ and that $H_2=\SL(V_2)$. Let $v\in V_1\otimes V_2$ have maximal rank. Then case (1) applies. If $(V_1,G_1,H_1,K_1)$ occurs in Table 2, let  $W_1$ be any nontrivial $K_1$-subspace of $V_1$ and let $v\in W_1\otimes V_2$ have maximal rank. Then case (2) applies.
\end{example}

From Proposition \ref{prop:g1} we get the following criterion for all nonzero orbits $Hv$ to be semi-characteristic.

\begin{corollary}\label{cor:allsemi}
Let $V$ be an irreducible $H$-module where $H$ is semisimple. Write $H=H_1\times\cdots\times   H_k$ where the $H_i$ are  simple for $i>1$, and let $V=V_1\otimes\dots\otimes V_k$ be the corresponding decomposition of $V$. Let $V''$ denote $V_2\otimes\dots\otimes V_k$ and set $H''= H_2\cdots H_k$. Suppose that none of the following occurs for any   decomposition $H=H_1\times\cdots\times H_k$.
\begin{enumerate}
\item There is a subordination $(V_1^*,H_1)\to (V'',H'')$ where $H_1\neq\SL(V_1)$.
\item There is a   tuple $(V_1,G_1,H_1,K_1)$  in Table 2  and  a subordination $(W^*,K_1)\to (V'',H'')$ where $W\subset V_1$ is $K_1$-stable.
\item There is a    tuple $(V_1,G_1,H_1,K_1)$  in Table 2  where $V_1^{K_1}\neq (0)$. 
\end{enumerate}
Then every nonzero $v\in V$ is semi-characteristic.
\end{corollary}

Admittedly, the corollary is a little unwieldy, but in any concrete case it is quite easy to apply. We see what we can say in the case of isotropy representations of symmetric spaces.

\begin{example}\label{ex:A5A1}
Let $H=\AA_5\times\AA_1$ acting on $V=\phi_3\otimes\phi_1$. This corresponds to the symmetric space of type EII (see \cite[Ch. X, Table V]{Helg78}). Let $0\neq v\in V$ and let $G$ be as usual with semisimple part $G_s$. If $G_s$ contains $H$, then $G_s$  cannot be simple (by Table 2), and if it is of the form $G_1\times G_2$ where $G_1\supset\AA_5$ and $G_1=\SL_2$, then it follows from Corollary \ref{cor:allsemi} or Proposition \ref{prop:g1} that $G_1=\AA_5$. Hence $v$ is semi-characteristic.
\end{example}

\begin{example}\label{ex:CpCq}
Let $p\geq q\in\N$ where $p>1$. Let $H=\Sp_{2p}\times\Sp_{2q}$ act in the natural way on $V=\C^{2p}\otimes\C^{2q}$. This corresponds to the symmetric space of type CII. Let $0\neq v\in V$. Then one easily sees that the only possibility for a semisimple $G_s$ containing $H$ stabilizing $Hv$ occurs in the case that $v$ has rank 1, in which case $G_s=\SL_{2p}\times\SL_{2q}$. For $q>1$ this corresponds to  Theorem \ref{thm:k=2}(1). If rank $v>1$, then $v$ is semi-characteristic.
\end{example}

\begin{example}\label{ex:glpq} Let $p\geq q\geq 1$. Let $H$ be the intersection of the block diagonal copy of $\GL_p\times\GL_q$ in $\GL_{p+q}$  with $\SL_{p+q}$. Then $H$ acts naturally on $V\oplus V^*$ where $V=\Hom(\C^p,\C^q)$. This is  an isotropy representation corresponding to the symmetric space of type AIII. First suppose that $q\geq 2$. Let $v=(x,x^*)$ where $x\in V$ and $x^*\in V^*$ are nonzero. Let $G=\{g\in\GL(V\oplus V^*)\mid gHv=Hv\}^0$.  Suppose that $v$ is not semi-characteristic. Then $\lieg$ is an $H$-stable Lie subalgebra of  $\Hom(V\oplus V^*,V\oplus V^*)$ which properly contains $\lieh$.
If $\lieg$ projected to $\Hom(V,V)$ or $\Hom(V^*,V^*)$ is more than a central extension of $\lieh$, then $\lieg$ has to contain $\lie{sl}_{p+q}$. But there is no corresponding entry in Table 2.   Thus we can suppose
that $\lieg$ projects nontrivially to one of the irreducible components of 
$$
\Hom (V,V^*)\simeq V^*\otimes V^*\simeq (S^2(\C^p)+\wedge^2(\C^p))\otimes(S^2((\C^q)^*)+\wedge^2((\C^q)^*))).
$$
Let us consider the case that $\lieg$ contains   $\lieg':=\wedge^2(\C^p)\otimes\wedge^2((\C^q)^*)$. Now $x$ has normal form $\sum_{i=1}^k e_i^*\otimes f_i$ where $e_1,\dots,e_p$ is a basis of $\C^p$, $f_1,\dots,f_q$ is a basis of $\C^q$ and $e_1^*,\dots, e_p^*$, $f_1^*,\dots,f_q^*$ denote  the elements of the dual bases. Then $e_1\wedge e_2\otimes f_1^*\wedge f_2^*$ lies in $\lieg'$  and applied to $x$ gives us $y^*=e_1\otimes f_1^*+e_2\otimes f_2^*$ if $k\geq 2$. The contraction of $x$ and $y^*$ (an $H$-invariant of $V\oplus V^*$) is not zero. Thus $(x,x^*+y^*)$ cannot be in the $H$-orbit of $v$, a contradiction. If $k=1$, then acting by the reductive part of $H_x$ we can bring $x^*$ to the normal form 
$$
ce_1\otimes f_1^*+e_1\otimes f^*+e\otimes f_1^*+\sum_{i=2}^\ell e_i\otimes f_i^*
$$
where $c\in \C$, $f^*\in \sspan\{f_2^*,\dots,f_q^*\}$ and $e\in \sspan\{e_2,\dots,e_p\}$. If $c\neq 0$, then acting by unipotent elements  of $H_x$ we can arrange that $e$ and $f^*$ are zero. Then  $\ell+1$ is an invariant of $x^*$ (its rank) under the action of $H_x$. But we can change $\ell$  by adding elements   of $\lieg'$ applied to $x$, again giving a contradiction.   If $c=0$ and $q>2$, then one similarly sees that we can change the rank of $x^*$. If $c=0$, $q=2$ and $e$ or $f^*\neq 0$, however, the $G'$-orbit of $v$ is contained in the $H_x$-orbit of $v$ and $v$ is not semi-characteristic.
The other three possible components of $\lieg$ give nothing new. Thus  $v$ is possibly not semi-characteristic only  when $q=2$,  $v$ is in the null cone and one of $x$ and $x^*$ has rank 1.

If $p=1$ and $q=1$ we have a torus action in which case $G=H$. If $q=1$ and $p\geq 2$  we  have the action of $\GL_p$ on $\C^p\oplus(\C^p)^*$. If $Hv$  is not closed, then $v$ is semi-characteristic. If $Hv$ is closed, then $G\simeq\SO_{2p}$ and $v$ is not semi-characteristic.
\end{example}

\begin{example}\label{ex:exceptions}
In general, one has a good chance to have points $v$ which are not semi-characteristic in case your representation is reducible. One can calculate that this actually occurs for the following isotropy representations of symmetric spaces.
\begin{enumerate}
\item $(V,H)=(\C^p\otimes\C^2,\SO_p\times\SO_2)$, $p\geq 3$. This is of type BDI and of type CI for $p=3$.
\item $(V,H)=(\wedge^2(\C^n)\oplus\wedge^2((\C^n)^*),\GL_n)$, $3\leq n\leq 5$. This is of type DIII.
\item $(V,H)=(\phi_4\otimes\nu_1+\phi_5\otimes\nu_{-1},\DD_5\times\C^*)$. This is of type EIII. Here $\nu_j$ denotes the one-dimensional representation of $\C^*$ of weight $j$.
\item $(V,H)=(\phi_1\otimes\nu_1+\phi_5\otimes\nu_{-1},\EE_6\times\C^*)$. This is of type EVII.
\end{enumerate}
\end{example}

Our discussion above establishes

\begin{proposition}
Let $(V,H)$ be the isotropy representation of an irreducible  symmetric space. Then, with the exception of the adjoint representation of $\CC_n$, $n\geq 2$ and the exceptions noted in Examples \ref{ex:CpCq}, \ref{ex:glpq} and  \ref{ex:exceptions}, every orbit $Hv$, $v$ generic, is semi-characteristic.
\end{proposition}

The proposition applies to some of the  questions of Ra\"is in \cite{Rais1}.

\section{Representations of SL$_2$}\label{sec:sl2}
 We consider the case of $H$-modules $V$ where $H:=\SL_2$ and $V^H=0$.  We have a generic $v\in V$  and $G:=\{g\in\GL(V)\mid gHv=Hv\}^0$ is not equal to $H$.   We denote by $R_n$ the $H$-module of binary forms of degree $n$. Then $R_n\simeq S^n(\C^2)$ has basis $x^n,x^{n-1}y,\dots,y^n$ where $x$, $y$ are the usual basis of $\C^2$ and $x^n$ is a highest weight vector.   Let  $N_G(H)$ denote the  connected normalizer of $H$ in $G$.

To determine $G$, we show that it suffices  to determine $N_G(H)$ and $\lieg_u$. We determine $N_G(H)$ in Theorem \ref{thm:sl2} below. We show that $\lieg_u$ is abelian and a multiplicity free $H$-module (Proposition \ref{prop:multfree}). We give necessary and sufficient conditions for $\lieg_u$ to contain a copy of $R_p$, $p>0$ (Theorem \ref{thm:mainsl2}). We then find some simple conditions that guarantee that $\lieg_u$ is zero or the trivial $H$-module for every generic $v\in V$ (Corollary \ref{cor:triv}).

  \begin{lemma}\label{lem:levi}
  Let $\tilde G$ be a Levi component of $G$ containing  $H$. Then $\tilde G\subset N_G(H)$.
   \end{lemma}
   
    \begin{proof}
   It follows from Theorem \ref{thm:g=h+k} that $\tilde G_v$ contains the simple components of $\tilde G$ of rank at least 2. These components are normalized by $H$  so they fix the whole orbit $Hv$ which spans $V$. Thus all the components of $\tilde G$ have rank at most 1 and $\tilde G\subset N_G(H)$. 
   \end{proof}

   \begin{corollary}\label{cor:normalizer}
   We have $\lieg\simeq \tilde\lieg\ltimes\lieg_u$. Hence $G\neq N_G(H)$ if and only if $\lieg_u$, as $H$-module, contains $R_p$ for some $p>0$. To determine $G$ it suffices to determine $N_G(H)$ and $\lieg_u$.
   \end{corollary}

We now consider the possibilities for $N_G(H)$.

  \begin{example}\label{ex:normal}
  \begin{enumerate}
  
\item Let $\bar H$ be another copy of $\SL_2$ and let $V_k=R_k\otimes \bar R_k$, $k\geq 1$, where $\bar R_k$ is the $\bar H$-module of binary forms of degree $k$.. Let $v_k\in V_k$ be a nonzero fixed vector of $\{ (h,h)\}\subset H\times \bar H$. Then $Hv_k=(H\times\bar H)v_k$ and $v_k$ is generic. We can also take   $V=V_{k_1}\oplus\dots\oplus V_{k_\ell}$ and $v=(v_{k_1},\dots,v_{k_\ell})$ where $k_1<\dots< k_\ell$. Then $v$ is generic and $N_G(H)\supset H \bar H$.

\item Let $V=\bigoplus_{k\in F} m_k R_k$ where   $1\leq m_k\leq k+1$ for all $k$ and $F$ is a nonempty finite subset of $\N$. Let $B$ and $\bar B$ be the standard Borel subgroups of $H$ and $\bar H$, respectively. Let $v_k$  be the highest weight vector  of the copy of $R_{2k+2-2m_k}$ in $V_k=R_k\otimes\bar R_k$ for the diagonal $H$-action. Then $v_k$  lies in $R_k$ tensored with the span of the weight vectors of $\bar R_k$ of weight at least   $k-2m_k+2$.  Now $v_k$ is an eigenvector for the diagonal copy of $B$, with weight $2k+2-2m_k$. For $\bar b\in \bar B$, let $\chi(\bar b)$ denote its upper left hand entry.  Let $\bar b$ act on $\bar R_k$ as the tensor product of the usual action and the scalar action  $\bar b\mapsto \chi(\bar b)^{2m_k-2k-2}$. Then $v_k$ is fixed by the diagonal in $B\times \bar B$. Assume that $m_k\geq 2$ for some $k$ so that   $\bar B$ acts effectively. Set $v=\bigoplus_{k\in F} v_k$. Then $v$ is generic and  $N_G(H)\supset H\bar B$.

\item Let $V=\bigoplus_{k\in F} m_k R_k$ as above.  Let $v\in V$ be a generic vector whose projection $v_{k,j}$ to the $j$th copy of $R_k$ is a  weight vector. If the weight is not zero, then there is an obvious $\C^*$-action on this copy of $R_k$ such that $v_{k,j}$ is fixed by the product of the standard torus in $H$ and our  external  copy of $\C^*$.  If $v$ is not a sum of zero weight vectors  we have  $N_G(H)\supset H\C^*$.

\end{enumerate}
  \end{example}

 \begin{theorem}\label{thm:sl2}
 Let $V=\bigoplus_{k\in F} m_k R_k$ be a representation of $H=\SL_2$ where $V^H=0$. Let $v\in V$ be generic.   If $N_G(H)\neq H$,   then, up to the action of $\prod_{k\in F}\GL_{m_k}$, we are  in one of the cases of Example \ref{ex:normal}. If $G\supset H \bar H$ as in Example \ref{ex:normal}(1), then $G=H\bar H$.
 \end{theorem}
  
  \begin{proof} 
 We have $N_G(H)=HG'$ where $G'$ is the identity component of the centralizer of $H$ in $G$. The group $G'_v$ fixes $Hv$, so it is trivial. Hence the Lie algebra of $G_v=\{hg'\mid hg'v=v\}$  projects onto $\lieg'$ and into $\lieh$, so that $G'$ is locally isomorphic to a quotient of a connected subgroup of $H$. Hence $G'$ is locally isomorphic to a connected subgroup of $H$.
  
  \smallskip
  
  \noindent Case 1:\ $G'=\SL_2$ or $\SO_3$.  Going to a cover, we can assume that $G'=\bar H$ so that $G_v$ is isomorphic to the diagonal copy of $H$. Then $V$ is a sum of representations $V_k=R_k\otimes S_k$ where $S_k$ is a representation of $\bar H$ of dimension at most $k+1$ and the projection of $v$ to $V_k$ is a fixed point of the diagonal action of $H$. Thus $S_k\simeq \bar R_k$ and $v$ is as in Example \ref{ex:normal}(1).   Suppose that $\lieg_u\neq 0$. 
  Then, as $(H\times\bar H)$-module, $\lieg_u$ cannot contain $R_0$ or $\bar R_0$ since the connected centralizer of $H$ is $\bar H$ and vice versa. Thus $\lieg_u$ contains a term $R_a\otimes\bar R_b$ where $ab\neq 0$. Hence, as $H$-module, $\lieg_u$ is not multiplicity free. But this contradicts Proposition \ref{prop:multfree} below. Hence $\lieg_u=0$ and $G=H\bar H$.
  
   \smallskip

  \noindent Case 2:\ $G'=\C^*$. Then $G_v\subset H\times\C^*$ is a diagonal torus and the fixed subspace of $G_v$ on each isotypic component $m_kR_k$ of $V$ is a sum of $m_k$ distinct weight spaces of $H$. Thus   we are in Example \ref{ex:normal}(3). 
  
    \smallskip
  
  \noindent Case 3:\   $G'\supset \bar U$ and $G'\not\supset \bar H$  where $\bar U\subset\bar H$ is the standard maximal unipotent subgroup of our second copy $\bar H$ of $\SL_2$.   We have $V=\bigoplus_{k\in F} R_k\otimes S_k$ where $S_k$ is a representation of $\bar U$. The isotropy group of $v$ in $H\times\bar U$ can be taken to be the diagonal copy of $U$ in $U\times\bar U$. Then $v$ corresponds to a subordination $S_k^*\to R_k$, hence the image of $S_k^*$ is a $U$-stable subspace of $R_k$ and $S_k$ is a $\bar B$-stable subspace of $\bar R_k$. 
  In fact, it is the span of $x^k$, $x^{k-1}y,\dots,x^{k-m_k+1}y^{m_k-1}$, and acting by   elements of the various $\GL(m_k)$ we can arrange  that $v$ is as in Example \ref{ex:normal}(2). Hence $N_G(H)\simeq H\bar B$.
\end{proof}

\begin{corollary}
Let $V$ be as above and $v\in V$  generic. Suppose that Example \ref{ex:normal}(1) does not apply so that  $N_G(H)\neq H\bar H$.
Then  $v$ is almost semi-characteristic.
\end{corollary}

We now turn to the determination of $\lieg_u$ when it is not zero or a trivial $H$-module. 

\begin{proposition}\label{prop:rp} Let $v\in V$ be generic.
Suppose that there is a copy of $R_p$ in $\lie{gl}(V)$,   $p>0$, which is a  Lie subalgebra and acts nilpotently on $V$. Further suppose that $R_p(v)\subset\lieh(v)$. Then $R_p\subset\lieg_u$.
\end{proposition}  

\begin{proof}  
Consider the action $\sigma\colon R_p\otimes V\to V$. Then for $h\in H$, 
$$
\sigma(R_p\otimes h(v))=h\sigma(R_p\otimes v)\subset h\lieh(v)=\lieh(hv).
$$
Hence $Hv$ is open in $G_pv$ where $G_p$ is the connected group with Lie algebra $\lieh\ltimes R_p$. Thus $G_p\subset G$ and $R_p\subset\lieg$. The projection of $R_p$ to $\Lie(N_G(H))$ is trivial (by our classification of $N_G(H)$ and the fact that $R_p$ is nilpotent). Hence $R_p\subset\lieg_u$.
\end{proof} 

\begin{remark}
The proposition remains true if $p=0$ as long as $G\neq H\bar H$ as in Example \ref{ex:normal}(1).
\end{remark}

 \begin{example}\label{ex:gu2=0}
 Let $p$, $l>0$. Let $v_{l+p}=x^{l+p}$ and let $v_l=a_0x^l+a_1x^{l-1}y$, $a_1\neq 0$.  Set $V=R_{l+p}+R_l$. Consider a nonzero equivariant map $\sigma\colon R_p\otimes R_ {l+p}\to R_l$. Then $\sigma(x^iy^{p-i}\otimes  v_{l+p})$ vanishes for $i>0$ and $\sigma(y^p\otimes v_{l+p})$ is a nonzero multiple of $x^l$. Thus we may arrange that $\sigma(y^p\otimes x^{l+p})=a_1x^l$. If $A\in\lie{sl}_2$ is $x\pt/\pt y$, then $\sigma(y^p\otimes v_{l+p})=A(v_l)$. We may consider $\sigma$ as an equivariant mapping of $R_p$ to $\Hom(R_{l+p},R_l)$.  Then $R_p$ applied to $v:=v_{l+p}+v_l$ is the same as $\lieu(v)$ where $\lieu=\C\cdot A$. By Proposition \ref{prop:rp}, $R_p\subset\lieg_u$. We can also have a copy of $R_q$ in $\lieg_u$, $q\neq p$, by adding   $R_{l+q}$ to $V$ and adding $v_{l+q}$ to $v$, where $v_{l+q}=x^{l+q}$.  
  \end{example}
 
We now try to pin down the structure of $V$ and $v$ The situation can be quite complicated. First we need a lemma.

       \begin{lemma}\label{lem:phineq0}
     Let $\phi\colon R_p\otimes R_n\to R_{p+n-2i}$ be equivariant and nonzero where $0\leq i\leq\min\{p,n\}$. Then $\phi(x^{p-j}y^j\otimes x^n)\neq 0$ for $i\leq j\leq p$.
     \end{lemma}
     
     \begin{proof} If $\phi(x^{p-j}y^j\otimes x^n)= 0$, then the $\lie{sl}_2$-submodule $W$ of $R_p\otimes R_n$ generated by $x^{p-j}y^j\otimes x^n$ lies in the kernel of $\phi$. Applying $x\pt/\pt y\in\lie{sl}_2$ repeatedly we may reduce to the case that $\phi(x^{p-i}y^i\otimes x^n)=0$. Suppose  by induction   that $x^{p-k}y^k\otimes x^{n-l}y^l$ lies in $W$ for $k+l=i$ and $l\leq s$. Then applying $y\pt/\pt x$ followed by $x\pt/\pt y$ to  $x^{p-k}y^k\otimes x^{n-s}y^s$ we obtain elements in $W$ as well as $k(n-s)x^{p-k+1}y^{k-1}\otimes x^{n-s-1}y^{s+1}$. Thus $W$ contains all the weight vectors of $R_p\otimes R_n$ of weight $p+n-2i$. This implies that $R_{p+n-2i}\subset W$, a contradiction. Thus $\phi(x^{p-j}y^j\otimes x^n)\neq 0$.
     \end{proof}
     
  \begin{remark}\label{rem:phineq0}
  By reversing the roles of $x$ and $y$ one has that $\phi(x^jy^{p-j} \otimes y^n)\neq 0$ for $i\leq j\leq p$.
   \end{remark}

    \begin{corollary}\label{cor:x^n}
   Let $\phi$, etc.\ be as above where $p+n-2i\neq 0$.  Let $w=x^n\in R_n$. Then $\dim\phi(R_p \otimes w)\geq 2$ unless $i=p<n$ so that  $\phi(y^p \otimes w)$ is a highest weight vector of $R_{n-p}$.  
  \end{corollary}
  
Set $W_0=V$ and for $j>0$ set $W_j=\lieg_u(W_{j-1})$. Then  $W_j$ is a proper $H$-stable subspace of $W_{j-1}$ for $j>0$. Let $k$ be the greatest  integer $j$ such that $W_j\neq 0$.  Since $\lieg_u$ acts nontrivially on $V$, we must have $k>0$. Let $V_j$ be an $H$-complement to $W_{j+1}$ in $W_j$ for $0\leq j\leq k$. Then $V=\oplus_j V_j$. Write $v=v_0+v_1+w_2$ where $v_i\in V_i$, $i=1$, $2$,  and $w_2\in W_2$. As before, let $A$ denote $x\pt/\pt y\in\lieh$.

     \begin{lemma}\label{lem:dimone} Perhaps replacing $v$ by $hv$ for some $h\in H$ we have the following.
     \begin{enumerate}
     \item  The vector $v_0$ is a sum of highest weight vectors.
     \item The dimension of $\lieg_u(v)$ is one with basis $A(v)$.
     \item Suppose that $R_p\subset \lieg_u$ where $p>0$. Then for $p\geq i>0$, $x^iy^{p-i}\in R_p$ annihilates $v$.
     \end{enumerate}
     \end{lemma}
     
     \begin{proof} Since $\lieg_u$ acts nontrivially on $V$ and $v$ is generic, there has to be a  $C\in\lieg_u$ such that $C(v)\in W_1$ and $C(v)\not\in W_2$. Then    there must be a $D\in\lieh$ such that $D(v_0+v_1)=C(v)$ modulo $W_2$. Since $D$ preserves $V_0$ and $V_1$, we must have that $D(v_1)=C(v_0)$ modulo $W_2$ and that $D$ annihilates $v_0$. Up to the action of $H$, we may thus assume that $v_0$ is a sum of highest weight vectors or a sum of zero weight vectors. We assume the latter and derive a contradiction. Since $\lieg_u$ is $H$-stable, we may assume that $C$ is a weight vector for the action of $\C^*\subset H$. Note that $D$ generates the Lie algebra of $\C^*\subset H$. If $C$ has weight zero, then so does $C(v)+W_2=C(v_0)+W_2$ and we cannot have that $C(v)=D(v)$ modulo $W_2$. Thus $C$ has weight $j$ for some $j\neq 0$ so that $C(v_0)+W_2=D(v_1)+W_2$ also has weight $j$. Hence $v_1=v_1'+v_1''$ where $v_1'+W_2=C(v_0)+W_2$ and $v_1'$ has weight $j$ while $v_1''$ has weight $0$.  Now let  $Z$ be the two-dimensional vector space generated by $v_0$ and $v_1'$, all modulo $W_2$. The groups generated by $\exp(tD)$ and $\exp(tC)$, $t\in\C$, act on $Z$ and the orbits of  $(v_0,v_1')$ are the same. But $\exp(tC)(v_0,v_1')$ contains the point $(v_0,0)$ while $\exp(tD)(v_0,v_1')$ clearly does not. Hence we have (1), i.e., $v_0$ is a sum of highest weight vectors. Moreover, $\lieg_u(v)+W_2$ is one-dimensional and generated by $A(v_1)+W_2$. 
           
           Let $C\in\lieg_u$ as above. Then $C(v)=D(v)$ for some $D\in\lieh_{v_0}$, where $D$ is a multiple of $A$. Hence we have (2). Finally, suppose that $R_p\subset \lieg_u$ where $p>0$. By Corollary \ref{cor:x^n}, for $i>0$, $x^iy^{p-i}$ annihlates $v$, modulo $W_2$, while $y^p$ sends $v$ to a multiple of $A(v)$, modulo $W_2$. If $x^iy^{p-i}$ acts nontrivially on $v$ it follows that $\dim\lieg_u(v)>1$. Hence we have (3).
      \end{proof}

      Let $\sigma\colon R_p\otimes  V\to V$ be the action of some $R_p\subset \lieg_u$ where $p\geq 0$.  Let $\mu\colon V\to V$ be the action of $y^p$ via $\sigma$. We may assume that $\mu(v)=A(v)$.  
\begin{corollary} 
\begin{enumerate}
\item  For all $j\geq 1$, $\mu^j(v)=A^j(v)$.
\item If $p>0$, then for all $1\leq i\leq p$, $ j\geq 0$, $\sigma(x^iy^{p-i}\otimes A^j(v))=0$.
\end{enumerate}
\end{corollary}

\begin{proof} Suppose that $p>0$. We prove (1) and (2) simultaneously by induction on $j$. Assume that $\mu^j(v)=A^j(v)$ for $1\leq j\leq m$ and that $\sigma(x^iy^{p-i}\otimes A^jv)=0$ for $0\leq j<m$, $i>0$. We certainly have the case that $m=1$. 
Apply $A$ to the equation $\sigma(y^p\otimes A^{m-1}(v))=A^m(v)$. Since $\sigma$ is equivariant, one obtains  that 
$$
\sigma(pxy^{p-1}\otimes A^{m-1}(v))+\sigma(y^p\otimes A^m(v))=A^{m+1}(v).
$$
Since the first term above is zero, we have that $\mu(A^m(v))=A^{m+1}v$ so that, by induction, we have $\mu^{m+1}(v)=A^{m+1}(v)$. Now apply $A$ to the equation $\sigma(x^iy^{p-i}\otimes A^{m-1}(v))=0$. One obtains that
$$
\sigma((p-i)x^{i+1}y^{p-i-1}\otimes A^{m-1}(v))+\sigma(x^iy^{p-i}\otimes A^m(v))=0
$$
so that $\sigma(x^iy^{p-i}\otimes A^m(v))=0$. This completes the induction. In case $p=0$, $A$ commutes with the generator of $R_0$, so that (1) is immediate.
\end{proof}

\begin{remark}\label{rem:induction}
Suppose that $p>0$ and that we  have (1) above. Then applying $A$ to the equations of (1) and using induction we obtain  (2).
\end{remark}
 
 \begin{proposition}\label{prop:multfree}
 The Lie algebra $\lieg_u$ is abelian and as $H$-module is multiplicity free.
 \end{proposition}
 
 \begin{proof}
 Suppose that we have copies of $R_p$ and $R_q$ in $\lieg_u$ where we allow $p=q$ (in which case we have two copies of $R_p$). If $[R_p.R_q]\neq 0$, then we have a copy of some $R_s$ in $\lieg_u$ which maps $V$ to $W_2$. Thus $R_s(v)\neq 0$ while $R_s(v)\in W_2$. This implies, as in the proof of Lemma \ref{lem:dimone}, that $\lieg_u(v)$ has dimension greater than one, a contradiction. Hence $\lieg_u$ is abelian. If $R_p$ has multiplicity two or more,  
 then it follows from Lemma \ref{lem:dimone} that there is a copy of $R_p$ which sends $v$ to $0$ implying that this copy of $R_p$ acts trivially on $V$, a contradiction. Hence $\lieg_u$ is multiplicity free.
  \end{proof}
  
  \begin{proposition}
  For all $i\geq 0$, $A^i(v)$ is generic in $W_i$.
  \end{proposition}
  
  \begin{proof}
  Since $v$ is generic in $V$ and $\lieg_u$ is $H$-stable, $W_1$ is generated by the $H$-orbit of $\lieg_u(v)$. Hence $Av$ is generic in $W_1$. Then the same argument shows that the $H$-orbit of $A^2(v)$ spans $W_2$, etc.
  \end{proof}
  We say that a vector $w\in R_l$ has \emph{height $k$\/} if $w=a_0x^l+\dots+a_kx^{l-k}y^k$ where $a_k\neq 0$. A vector in $Z:=\sum_i m_i R_i$ has \emph{height at least $k$\/} (resp.\ \emph{height at most $k$\/}) if it is generic in $Z$ and when written as a sum $\sum_i v_{i,1}+\dots+v_{i,m_i}$ where $v_{i,j}$ is in the $j$th copy of $R_i$, each $v_{i,j}$ has height at least $k$ (resp.\ at most $k$).
  
  \begin{proposition}
   The $H$-modules $V_i$ are multiplicity free.
 \end{proposition}
  
  \begin{proof} 
 The vector $A^jv$ is generic in $W_j$, $j\geq 0$, and the projection of $A^jv$ to any $R_l$ in $W_j$ cannot be zero.  Thus the projection of $v$ to any $R_l\subset W_j$ has height at least $j$. We have $v+W_j=v_0+v_1+\dots+v_{j-1}+W_j $ where $A^jv\in W_j$. It follows that $A^jv_i=0$ for $i<j$, hence any $v_i$ is a sum of vectors of height at most $i$. Since $v_j\in W_j$ it is a sum of vectors of height at least $j$. Thus $A^jv_j$  is a sum of highest weight vectors and it is generic in $W_j$. Hence any $R_l$ can occur in $W_j$ with multiplicity at most one.
     \end{proof}

          Write $v=v_0+\dots+v_k$ where $v_i\in V_i$. Then each $v_i$ is a sum $\sum_{l\in F_i} v_{i,l}$ where $F_i\subset \N$ and   $v_{i,l}$ lies in the copy of $R_l\subset V_i$.
       \begin{corollary}
Each vector $v_{i,l}$ has height $i$.
       \end{corollary}     
      
     \begin{lemma}
     Let $\phi\colon R_p\otimes R_m\to R_l$ be equivariant and nonzero.
     \begin{enumerate}
\item Necessarily $m=l+p-2i$ for some $i$ with $0\leq i\leq\min\{l,p\}$.
\item Suppose that $w\in R_m$ has height $n\leq l-i$. Then $\phi(y^p\otimes w)$ has height $n+i$.
\end{enumerate}
\end{lemma}

\begin{proof}
Since representations of $H$ are self-dual, $R_l$ appears in $R_p\otimes R_m$ if and only if $R_m$ appears in $R_p\otimes R_l$. Then Clebsch-Gordan implies (1). Now consider $z:=\phi(y^p\otimes x^{m-n}y^n)$ where $m=l+p-2i$. Then Remark \ref{rem:phineq0} shows that $z\neq 0$ if the weight of $y^p\otimes x^{m-n}y^n$ is at least $-l$. This is equivalent to $n\leq l-i$, hence we have (2).
\end{proof}

As above, we have  $v_i=\sum_{l\in F_i} v_{i,l}$ where $v_{i,l}\in R_l\subset V_i$. For any $s\geq 0$, we have $W_1=V_1\oplus V_2\oplus\dots\oplus V_{s+1}\oplus W_{s+2}$, hence we have an $H$-equivariant projection of $W_1$ to $V_{s+1}$. Let $\tau$ denote $\sigma$ on $R_p\otimes(V_0+\dots+V_{s})$ followed by projection onto $R_l\subset V_{s+1}$.
Since we have $\sigma(y^p\otimes A^r(v))=A^{r+1}(v)$, $r\geq 0$, for every $v_{s+1,l}\in R_l\subset V_{s+1}$, $A^{r+1}(v_{s+1,l})$ must be a multiple of $\tau(y^p\otimes A^r(v_0+\dots+v_{s}))$ for $r\geq 0$. Note that $\tau$ vanishes on $R_p\otimes v_{i,t}$ unless  $t=l+p-2j$ where $0\leq j\leq\min\{p,s\}$.  
\begin{proposition}
Let $s$ and $l$  be as above. Let $R_{l+p-2j}\subset  V_i$, $j\leq\min\{p,s\}$. If $i+j> s$, then $\tau(R_p\otimes R_{l+p-2j})=0$.
\end{proposition}

\begin{proof}
Consider the pairs $(i',j')$ where $0\leq i'\leq s$, $0\leq j'\leq\min\{p,s\}$, $i'+j'> s$ and $v_{i',l+p-2j'}\neq 0$. Assume that $i$ is the maximal $i'$ that occurs and that $j$ is the maximal $j'$ that occurs in a pair $(i,j')$. Consider $A^i(v_{i,l+p-2j})$. It is a highest weight vector of weight $l+p-2j$. Suppose that    $\tau(y^p\otimes A^i(v_{i,l+p-2j}))$ is nonzero. Then it has height $j> s-i$. Moreover, by the choice of $i$ and $j$, $\tau(y^p\otimes A^i(v_{i,l+p-2j}))$ is the  nonzero $\tau(y^p\otimes A^i(v_{i,l+p-2j'}))$ of largest height (equivalently, of lowest weight). But $\tau(y^p\otimes A^i(v))=A^{i+1}(v_{s+1,l})$  where $A^{i+1}(v_{s+1,l})$ has height $s-i$. Thus $\tau(y^p\otimes A^i(v_{i,l+p-2j}))$ must be zero. Now for $0<m\leq p$  we have that $\sigma(x^my^{p-m}\otimes A^i(v))=0$. Again, by height considerations, one sees that $\tau(x^my^{p-m}\otimes A^i(v_{i,l+p-2j}))$ must vanish. Hence $\tau(R_p\otimes A^i(v_{i,l+p-2j}))=0$ which  shows that $\tau(R_p\otimes R_{l+p-2j})=0$.  Now the proof can be completed by downward induction on $i$ and $j$.
\end{proof}

For $0\leq j\leq\min\{p,s\}$ and $R_{l+p-2j}\subset V_i$, the restriction of $\tau$ to $R_p\otimes R_{l+p-2j}$ is a multiple  $t_{i,s-j}\tau_{s-j}$ of $\tau_{s-j}$ where $\tau_{s-j}\colon R_p\otimes R_{l+p-2j}\to R_l$ is equivariant and normalized so that $\tau_{s-j}(y^p\otimes x^{l+p-2j})=x^{l-j}y^j$.
We have that 
\begin{equation}\label{eq:basic}
\sum_{j=0}^{\min\{p,s\}}\sum_{i=0}^{ s-j}  t_{i,s-j}\tau_{s-j}(y^p\otimes A^r(v_{i,l+p-2j}))=A^{r+1}(v_{s+1,l})
\end{equation}
for all $r\geq 0$ where we set $v_{i,q}=0$ if $q\not\in F_i$.
 
Let $a_i$ and $b^{s-j}_{i,m}$ be scalars such that $v_{s+1,l}=a_0x^l+\dots+ a_{s+1}x^{l-s-1}y^{s+1}$ and $v_{i,l+p-2j}=b^{s-j}_{i,0} x^{l+p-2j}+\dots+b^{s-j}_{i,i}x^{l+p-2j-i}y^i$ whenever $l+p-2j\in F_i$ and $0\leq j\leq\min\{p,s\}$. We use $t_j$ and $b_j$ as shorthand for $t_{j,j}$ and $b^j_{j,j}$, respectively. For now assume that $l+p-2j\in F_{s-j}$, $j=0,\dots,\min\{p,s\}$.
 
 \begin{theorem}\label{thm:maineqn}  
For  $m=0,\dots,s$, consider the equation  \eqref{eq:basic} with $r=m$ in weight $l-2s+2m$. This gives us $s+1$ equations
 in the unknowns $t_{s-j}b_{s-j}$ for $0\leq j\leq\min\{p,s\}$.
 The unique solutions are 
 $$
 t_{s-j}b_{s-j}=\frac{\binom sj\binom pj}{\binom{l+p-j+1}{j}}(s+1)a_{s+1},\ 0 \leq j\leq \min\{p,s\}.
 $$
 \end{theorem}

 \begin{proof} First assume that $p\geq s$.
 Since $\tau_{s-j}(y^p\otimes x^{l+p-2j})=x^{l-j}y^j$, it follows that 
 $$
 \tau_{s-j}(y^p\otimes x^{l+p-2j-k}y^k)=\frac{\binom{l-j}k}{\binom{l+p-2j}k}x^{l-j-k}y^{j+k}, \ k\leq l-j.
 $$
 Now the $0$th equation ($m=0$) is
 $$
 t_0b_0+t_1b_1\tfrac{l-s+1}{l+p-2s+2}+\dots+t_jb_j\frac{\binom{l-s+j}j}{\binom{l+p-2s+2j}j}+\dots+t_sb_s\frac{\binom ls}{\binom {l+p}s}=(s+1)a_{s+1}.
 $$
 For $m=1$ the equation is
 $$
 t_1b_1+2t_2b_2\tfrac{l-s+2}{l+p-2s+4}+\dots+jt_jb_j\frac{\binom{l-s+j}{j-1}}{\binom{l+p-2s+2j}{j-1}}+\dots+st_sb_s\frac{\binom l{s-1}}{\binom {l+p}{s-1}}=s(s+1)a_{s+1}
  $$
  and the $m$th equation is
  $$
  m!t_mb_m+\dots+j!/(j-m)!t_jb_j\frac{\binom{l-s+j}{j-m}}{\binom{l+p-2s+2j}{j-m}}+\dots+\frac{s!}{(s-m)!}t_sb_s\frac{\binom l{s-m}}{\binom{l+p}{s-m}}=\frac{(s+1)!}{(s-m)!}a_{s+1}. 
  $$
 Thus our system of equations  is equivalent to
\begin{equation}\label{eq:star}
  \sum_{j=0}^s\binom jm c_j\frac{\binom{l-s+j}{j-m}}{\binom{l+p-2s+2j}{j-m}}=\binom sm,\ m=0,\dots,s
\end{equation}
where $c_j=t_jb_j/((s+1)a_{s+1})$. Since the equations are in triangular form, there is a unique solution. Now the theorem will be proved if we can show that  a  solution  to \eqref{eq:star} is
$$
c_{s-j}=\frac{\binom sj\binom pj}{\binom{l+p-j+1}{j}}.
$$ 
But one can prove this using the  WZ  method \cite{WZalgorithm}. (See  \cite{WZnotices} for a brief introduction.)\  We used the implementation of the WZ method in MAPLE.

Now suppose that $p<s$. We may still consider the system of equations \eqref{eq:star}. The solutions remain the same, but note that for $j>p$, the formula for $c_{s-j}$ gives zero.   Hence the theorem is true even when $p<s$.
 \end{proof}
 
  \begin{remark}  
 Let $0\leq j\leq \min\{p,s\}$. We assumed that $R_{l+p-2j}$ occurred in $V_{s-j}$. But the equations force $t_{s-j}b_{s-j}$ to be nonzero. Thus, in fact, $R_{l+p-2j}$ must occur in $V_{s-j}$ for there to be a solution of \eqref{eq:basic} for all $r\geq 0$.
 \end{remark}
 
 Since we are guaranteed to have vectors $v_{s-j,l+p-2j}$ in our solution of \eqref{eq:basic}, what role do the vectors $v_{i,l+p-2j}$ play for $i<s-j$? It is easy to see that the term involving $v_{i,l+p-2j}$ may be eliminated if we change $v_{s-j,l+p-2j}$ to $v_{s-j,l+p-2j}+\frac{t_{i,s-j}}{t_{s-j}}v_{i,l+p-2j}$.    Let us say that $v''=\sum_j v'_{s-j,l+p-2j}$ is obtained from $v'=\sum_j v_{s-j,l+p-2j}$ by an \emph{admissible modification\/} if each $v'_{s-j,l+p-2j}$ differs from $v_{s-j,l+p-2j}$ by a linear combination of the $v_{i,l+p-2j}$ for $i<s-j$. Thus we have the following 

 \begin{remark}
 We have a solution of \eqref{eq:basic} if and only if,  up to an admissible modification of the $v_{s-j,l+p-2j}$, we have a solution of
 \begin{equation}\label{eq:basic2}
 \sum_{j=0}^{\min\{p,s\}}t_{s-j}\tau_{s-j}(y^p\otimes A^r(v_{s-j,l+p-2j}))=A^{r+1}(v_{s+1,l}),\ r\geq 0.
 \end{equation}
 \end{remark}
  
 \begin{proposition}\label{prop:unique}
Let $v_{s+1,l}\in V_{s+1}$   and $v_{s-j,l+p-2j}\in V_{s-j}$ have coefficients $a_j$ and $b^{s-j}_{i,m}$ as above. Fix the $b_{s-j}$, $0\leq j\leq\min\{p,s\}$. Then there are unique values of the $t_{s-j}$ and $b^{s-j}_{s-j,m}$    for $m<s-j$ such that there is a solution of \eqref{eq:basic2}.
 \end{proposition}
 
 \begin{proof}
 We know that the $t_{s-j}$ are uniquely determined. We only need to show that   the $b^{s-j}_{s-j,m}$ for $m<s-j$ are unique satisfying \eqref{eq:basic2}. This is easy because of the triangular form of the equations. For $r=0$,   the equation in weight $l$ reads $t_sb^s_{s,0}=a_1$. For arbitrary $r\leq s$, the equation in weight $l$ is $r!t_sb^s_{s,r}=(r+1)!a_{r+1}$.  For $r=s$ this is one of the equations we considered in Theorem \ref{thm:maineqn} and we have $b^s_{s,r}=\frac{r+1}{t_{s}}a_{r+1}$ for $r<  s$. Now suppose that we have determined the $b^{s-q}_{s-q,j}$ for $0\leq q<m$. Consider \eqref{eq:basic2} in weight $l-2m$ with $r=0$. It gives an expression for $t_{s-m}b^{s-m}_{s-m,0}$ in terms of the $a_i$ and $b^{s'}_{s',s'-j'}$ for $s'>s-m$. Thus we may solve for $b^{s-m}_{s-m,0}$. For $0<r\leq s-m$ we obtain an equation that we can solve for $b^{s-m}_{s-m,r}$. The equation that we get for $b_{s-m}$ is one of the equations that we considered in Theorem \ref{thm:maineqn}. Hence given $a_1,\dots,a_{s+1}$ and the $b_{s-j}$, there are unique $t_{s-j}$ and $b^{s-j}_{s-j,m}$ solving \eqref{eq:basic2}.
  \end{proof}
  
   \begin{remark}\label{rem:eqn}
 Suppose that $l+p-2j\in F_i$ for all $i+j\leq s$, $0\leq j\leq\min\{p,s\}$. Then we may modify the $v_{s-j,l+p-2j}$ admissibly so that the $b^{s-j}_{s-j,m}$, $m<s-j$, are arbitrary. Hence there are $t_{s-j}$ giving solutions  of \eqref{eq:basic2} (after admissible modifications) and giving solutions of   \eqref{eq:basic} (without  changing any vectors).
 \end{remark}

 Let us formulate the conditions that need to be satisfied to have $R_p\subset\lieg_u$, $p>0$.
 \begin{definition}
 Let $v\in V$ be generic. We say that \emph{$v$ satisfies $(*_p)$\/} if 
 \begin{enumerate}
\item We have a decomposition $V=\oplus_{i=0}^k V_i$ where the $V_i$ are multiplicity free $H$-modules. Let $F_i\subset \N$ such that $V_i=\oplus_{l\in F_i} R_l$.
\item Possibly replacing $v$ by $hv$ for some $h\in H$, we have that $v=\sum_i\sum_{l\in F_i} v_{i,l}$ where $v_{i,l}\in R_l\subset V_i$ has height $i$.  
\item For every $v_{s+1,l}\in V_{s+1}$, $s\geq 0$, we have that $l+p-2j\in F_{s-j}$ for $0\leq j\leq\min\{p,s\}$. Let the $t_{s-j}$ be given by  Theorem \ref{thm:maineqn}. Then the vectors $v_{s-j,l+p-2j}$, perhaps after admissible modification, are solutions of  \eqref{eq:basic2}.
\end{enumerate}
 \end{definition}

 \begin{theorem}\label{thm:mainsl2}
Let $v\in V$ be generic. Then $R_p\subset\lieg_u$, $p>0$, if and only if $v$ satisfies $(*_p)$. 
 \end{theorem}
 
 \begin{proof}
 We have shown that $R_p\subset\lieg_u$ implies that $(*_p)$ holds. Conversely, if $(*_p)$ holds, then we have constants $t_{i,s-j}$ such that \eqref{eq:basic} is satisfied. Let $\tau$ denote the corresponding map  
 $$
R_p\otimes ( \bigoplus_{j=0}^{\min\{p,s\}}\bigoplus_{i=0}^{ s-j}\bigoplus_{l+p-2j\in F_i} R_{l+p-2j}\subset V_i) \to R_l\subset V_{s+1}.
 $$
The various mappings $\tau$ combine  to give us an equivariant map  $\sigma\colon R_p\otimes V\to V$. It follows from Remark \ref{rem:induction} that $\sigma(x^iy^{p-i}\otimes v)=0$ for $i>0$. By construction,  $\sigma(R_p\otimes v)$ is one-dimensional and generated by $\sigma(y^p\otimes v)=A(v)$. 
If $S$ denotes the  copy of $R_p\subset\End(V)$ corresponding to $\sigma$, then $[S,S](v)=0$ which implies that $[S,S]$ acts trivially on $V$, i.e., $[S,S]=0$. By construction, $S$ consists of nilpotent transformations. Now 
by Proposition \ref{prop:rp} we have $S\subset \lieg_u$.
 \end{proof}

\begin{corollary}\label{cor:condition}
Suppose that there is a generic $v\in V$   such that $\lieg_u$ is not zero or the trivial $H$-module. Then there are subsets $F_0,\dots,F_k\subset\N$ such that $V=\bigoplus_{i=0}^k\bigoplus_{l\in F_i} R_l$ and $p>0$ such that  for every $l\in F_{s+1}$, $s\geq 0$ we have  $l+p-2j\in F_{s-j}$ for $0\leq j\leq \min\{p,s\}$.
\end{corollary}

\begin{corollary}\label{cor:triv}
The group $G$ normalizes $H$ if $V$ does not satisfy the condition of Corollary \ref{cor:condition}. In particular, $G$ normalizes $H$ in the following cases.
\begin{enumerate}
\item $V$ is an isotypic $H$-module.
\item The multiplicity of $R_l$ is at least two, where $l$ is maximal such that   $R_l\subset V$. 
\end{enumerate}
\end{corollary}

\begin{proof}
Part (1) is clear. In case (2), there has to be a vector $v_{s+1,l}$ where $s\geq 0$. Thus we must have $R_{p+l}\subset V_{s}$, which obviously fails.
\end{proof}

Using Remark \ref{rem:eqn} it is clear that one can have  extremely complicated situations where $R_p\subset\lieg_u$. Here is a modestly complicated case.
 
 \begin{example}\label{ex:complicated}
 Let $V_0=R_{l+p-2}\oplus R_{l+2p}$, $V_1=R_{l+p}$ and $V_2=R_l$, where $l>1$, $p>0$. Let $v=v_{0,l+p-2}+v_{0,l+2p}+v_{1,l+p}+v_{2,l}\in V=V_0\oplus V_1\oplus V_2$ where the $v_{r,s}$ are of height $r$ in $R_s\subset V_r$. Then by Remark  \ref{rem:eqn}, $\lieg_u$ contains a copy of $R_p$. Here  we have that $\sigma(y^p\otimes v_{0,l+2p})=A(v_{1,l+p})$ and $\sigma(y^p\otimes A^r(v_{0,l+p-2}+v_{1,l+p}))=A^{r+1}(v_{2,l})$, $r=0$, $1$. If we add a copy of $R_l$ to $V_1$ and a copy of $R_l$ to $V_0$ (assume $p\neq 2$) with corresponding components $v_{1,l}$ and $v_{0,l}$ in $v$, then we also have a copy of $R_0$ in $\lieg_u$. If $p=2$, we already have $R_l\subset V_0$ and we  only have to add $R_l\subset V_1$ and $v_{1,l}$. 
  \end{example}
  
  \begin{example}
  Suppose that $V=2R_l\oplus R_{l-1} \oplus R_{l+1}$ where $l\geq 2$. Then it is possible to have a generic $v\in V$ such that $R_1\subset \lieg_u$. However, one can check that this is not possible if we increase the multiplicity of $R_l$ to $3$.
  \end{example}

  \section{Appendix}  

Here we establish the branching rules which are used in Table 2 and the calculation of $V^K$ in cases 6.3 and 6.4. Recall that if $\phi$ is a $G$-module, then $\ss(\phi)=\oplus_k\ss^k(\phi)$ where $
\ss^k(\phi)$ denotes the subspace of $S^k(\phi)$ obtained using Cartan multiplication of the irreducible subrepresentations of $\phi$. 

Let $n=2m+1$, $m\geq 1$. Let $X$ denote the cone in $\phi_{n-1}(\DD_n)$ which is the closure of the orbit of a highest weight vector. Consider the action of $\SL_n\times\C^*$ on $\O(X)$ where $\phi_1(\DD_n)=\C^{2n}=\phi_1(\SL_n)\oplus\phi_{n-1}(\SL_n)=\C^n\oplus(\C^n)^*$. Here  $\C^n$ is the span of the positive weight vectors of $\DD_n$, $\C^*$ acts on $\C^n$ with weight $2$ and on $(\C^n)^*$ with weight $-2$. As $\SL_n\times\C^*$ representation, $\phi_{n-1}(\DD_n)$ is $\nu_n\oplus\phi_{n-2}(\SL_n)\otimes\nu_{n-4}+\phi_{n-4}(\SL_n)\otimes\nu_{n-8}+\dots+\phi_1(\SL_n)\otimes\nu_{-n+2}$  and $\phi_n(\DD_n)=\phi_{n-1}(\DD_n)^*=\phi_{n-1}(\SL_n)\otimes\nu_{n-2}+\dots+\phi_2(\SL_n)\otimes\nu_{-n+4}+\nu_{-n}$.

\begin{theorem} Let $X$ be the closure of the highest weight orbit in $\phi_{n-1}(\DD_n)$. Then, as $(\SL_n\times\C^*)$-module,    $\O(X)=\ss(\phi_n(\DD_n))$.
\end{theorem}

\begin{proof}
It is well-known that $X$ is normal, and every point of $X$ except the origin is smooth. For $x\in\nu_n$, $x\neq 0$, $x$ is a smooth point of $X$, $x$ is a fixed point of $\SL_n$ and the slice representation of $\SL_n$ at $x$ is $\theta_1+\phi_{n-2}$. Since $\phi_{n-2}$ has no invariants,  $\O(X)^{\SL_n}$ is generated by a coordinate function $z$ on $\nu_n$. Then $z$ is not a zero divisor in $\O(X)$, hence $\O(X)$ is a free $\C[z]$-module. By Luna's slice theorem \cite{Luna73}, $\O(X)$ is a free  $\C[z]$-module on $\O(\phi_{n-2})$.    But $\O(\phi_{n-2})=S(\phi_2)$ is just the sum of all representations of the form $\phi_2^{a_2}\cdot\phi_4^{a_4}\dots\phi_{n-1}^{a_{n-1}}$ each with multiplicity one. It follows that  the products of the highest weight vectors of the restriction of  $\phi_{n-1}(\DD_n)^*$ to $\SL_n$ freely generate the highest weights of $\O(X)$ as an $\SL_n$-module and as an $(\SL_n\times\C^*)$-module.
\end{proof}

Now suppose the $n=2m$, $m\geq 2$. Let $X$ denote the closure of the orbit of a highest weight vector of $\phi_{n-1}(\DD_n)$. Consider the action of $\SL_n\times\C^*\subset\DD_n$ such that  $\phi_1(\DD_n)$  becomes $\C^n\otimes\nu_1\oplus (\C^n)^*\otimes\nu_{-1}$. Effectively, we have the action of $\GL_n$. Then $\phi_{n-1}(\DD_n)$, as a  $\GL_n$-module, is $\nu_m+\phi_{n-2}\otimes\nu_{m-2}+\dots+\nu_{-m}$.

\begin{theorem}
Let $n=2m$, $m\geq 2$ and let $X$ be the closure of the highest weight orbit in $\phi_{n-1}(\DD_n)$. Then, as $\GL_n$-module,    $\O(X)=\ss(\phi_{n-1}(\DD_n))$.
\end{theorem}

\begin{proof} Let $z_{\pm}$ be coordinate functions on the copies of $\nu_{\pm m}$ in $\phi_{n-1}(\DD_n)$. As above, one computes that
there is a slice representation $(\phi_{2m-2}+\theta_1,\SL_{2m})$ for the action of $\SL_{2m}$ on $X$. The slice representation   has a quotient of dimension two  and principal isotropy group   $\CC_m$.   It follows that  the $\GL_n$-invariants have dimension 1, hence they must be  generated by $z_+z_-$. Moreover, the only way that the trivial $\SL_n$-representation can occur in $\C[\phi_{n-2}\otimes\nu_{m-2}+\dots+\phi_2\otimes\nu_{-m+2}]$ is in products whose $\C^*$-weight is a multiple of $\pm m$ (just count boxes in Young diagrams). Since $\GL_n$ is spherical in $\DD_n$, each $\nu_{km}$, $k\in \Z$, occurs once in the free $\C[z_+z_-]$-module $\O(X)$. Thus the $\SL_n$-invariants must be the  polynomial ring $\C[z_+,z_-]$ and $\O(X)$ is free over $\C[z_+,z_-]$. For the corresponding map $X\to\C^2$, the general fiber is $SL_n/\CC_m$, which gives that the only $\SL_n$-representations that occur are  $\phi_2^{a_2}\dots\phi_{n-2}^{a_{n-2}}$ for $a_2,\dots,a_{n-2}\geq 0$, each with multiplicity one. It follows that $\O(X)=\ss(\phi_{n-1}(\DD_n))$.
\end{proof}

Finally, we consider the case where $X$ is the closure of the highest weight vector in $\phi_n(\DD_n)$, $n=2m\geq 4$. As $\GL_n$-module, we have $\phi_n(\DD_n)=\phi_{n-1}\otimes\nu_{m-1}\oplus\dots\oplus\phi_{1}\otimes\nu_{-m+1}$.

\begin{theorem}
As $\GL_n$-module, $\O(X)=\ss(\phi_n(\DD_n))$.
\end{theorem}

\begin{proof}
There are no invariants in this case, so we have to proceed a little differently. We first find a general point of $X$. Let $e_1,\dots,e_{n}$ be the usual basis of $\C^{n}$. 
Let $\omega=e_2\wedge e_3+\dots+e_{2n-2}\wedge e_{2n-1}$ considered as an element of the Lie algebra of $\DD_n$. Then the action of $\exp(\omega)$ on $e_1$ sends it to the sum $v$ of the elements $e_1\wedge\omega^{k}\in\phi_{2k+1}$, $k=0,\dots,m-1$. The isotropy group $H$ of  $\SL_n$ acting on $v$ is the semidirect product of $\CC_{m-1}$ with $\Hom(\C \cdot e_n,\C^{n-1})\oplus\Hom(\C^{n-2},\C\cdot e_1)$ where $\C^{n-2}$ here stands for the span of $e_2,\dots,e_{n-1}$ and $\C^{n-1}$ stands for $\C^{n-2}\oplus\C\cdot e_1$. Note that our copy of $\CC_{m-1}$ acts standardly on $\C^{n-2}$. Now $\dim SL_n/H = \dim X$, so that $\SL_n\cdot  v$ is a dense orbit in $X$.    Since $X$ is factorial \cite[Theorem 4]{PopVin72},  
 any divisor in the  complement of the dense orbit must be defined by a semi-invariant of $\SL_n$, hence by an invariant.  Thus there are no such divisors, so that the complement of $\SL_n\cdot v$ has codimension 2. It follows that $\O(X)\simeq\O(\SL_n/H)$. But the irreducibles of $\SL_n$ with an $H$-fixed vector are those of the form $\phi_1^{a_1}\phi_3^{a_3}\dots\phi_{n-1}^{a_{n-1}}$ where the $a_i$ are nonnegative, and the fixed point set has dimension one. Thus $\O(X)$ is as claimed.
\end{proof}

We now compute the ring of $K$-invariants in the cases  (6.3)  and (6.4)  of Table 2.
\begin{proposition}
Let $X$ (resp.\ $Y$) be the closure of the orbit of the highest weight vector of $\phi_7(\DD_8)$ (resp.\ $\phi_8(\DD_8)$). Consider the action of $\BB_4$ on $X$ and $Y$ where $\phi_1(\DD_8)|\BB_4=\phi_4(\BB_4)$. Then $\O(X)^{\BB_4}=\C[f_4]$ and $\O(Y)^{\BB_4}=\C[f_2,f_3]$ where $\deg f_i=i$.
\end{proposition}

\begin{proof}  Using LiE \cite{Lie, Lie2} one computes that the Poincar\'e series of $\O(X)^{\BB_4}$ is $1+t^4+\dots$ and that the Poincar\'e series of $\O(Y)^{\BB_4}$ is $1+t^2+t^3+t^4+t^5+\dots$. Recall that $X$ and $Y$ are normal, hence so are $\O(X)^{\BB_4}$ and $\O(Y)^{\BB_4}$. Thus $\dim\O(Y)^{\BB_4}\geq 2$. The restriction of $\phi_7(\DD_8)$ (resp.\ $\phi_8(\DD_8)$) to $\BB_4$ is $\phi_1\phi_4$ (resp.\ $\phi_1^2+\phi_3$). Let $P$ (resp.\ $Q$) be the stabilizer of the highest weight line in $\phi_7(\DD_8)$ (resp.\ $\phi_8(\DD_8)$). Then   the Levi components $L(P)$ and $L(Q)$ of $P$ and $Q$  double cover representatives of the two $\SO_{16}$-conjugacy classes of embeddings of $\GL_8$ in $\SO_{16}$. We have $L(P)\simeq(\SL_8\times\C^*)/(\Z/4\Z)$ and the same for $L(Q)$. Restricted to  $L(P)$, $\phi_7(\DD_8)$ becomes  the representation $\nu_4+\wedge^6(\C^8)\otimes\nu_2+\wedge^4(\C^8)+\wedge^2(\C^8)\otimes\nu_{-2}+\nu_{-4}$.  The highest weight space of $\phi_7(\DD_8)$ is $\nu_4$. The tangent space to $X$ at a nonzero point of $\nu_4$ is $\nu_4+\wedge^6(\C^8)\otimes\nu_2$ so that $\dim X=29$. The restriction of $\phi_7(\DD_8)$ to $L(Q)$ is  $\wedge^7(\C^8)\otimes\nu_{3}+\wedge^5(\C^8)\otimes\nu_1+\wedge^3(\C^8)\otimes\nu_{-1}+\C^8\otimes\nu_{-3}$. For $\phi_8(\DD_8)$, the decompositions relative to $L(P)$ and $L(Q)$ are reversed, so $\dim Y=29$, also.

Consider the action of $H=\Ad\SL_3$ on $\C^9$ as $\phi_1\phi_2+\theta_1$. Then $\phi_4(\BB_4)|H=2\phi_1\phi_2$.  Clearly the image of $H$ in $\SO_{16}$ lies in a copy of $\GL_8$. Suppose that this copy of $\GL_8$ is double covered by a conjugate of $L(P)$. Then $X^{H}\neq  (0)$, and the $\BB_4$-orbit of a nonzero fixed point is closed since the normalizer of $H$ in $\BB_4$ is a finite extension of $H$ \cite[3.1 Corollary 1]{Luna75}. It is easy to check that the isotropy group of a nonzero point of $X^H$ is at most a finite extension of $H$. Thus the dimension of the corresponding closed $\BB_4$-orbit is $28$. Hence  $\dim\O(X)^{\BB_4}\leq 1$ and the Poincar\'e series information gives that $\O(X)^{\BB_4}=\C[f_4]$ where $\deg f_4=4$.  If our copy of $\GL_8$ were double covered by a conjugate of $L(Q)$, then we would see that $\dim \O(Y)^{\BB_4}\leq 1$, which is a contradiction. Thus $\O(X)^{\BB_4}$ is as claimed.

Now consider the group $K=\SO_6\times\SO_3\subset\SO_9$. Then the double cover $\tilde K$ of $K$ is  $(\SL_4\times\SL_2)/\pm I$ and $\phi_4(\BB_4)$, as $\tilde K$-representation, is $\C^4\otimes\C^2+(\C^4)^*\otimes\C^2$. Thus $\tilde K$ is a subgroup of a copy of $\GL_8$ in $\SO_{16}$. If this $\GL_8$ is double covered by a conjugate of $L(P)$, then one sees that there are no nonzero fixed points of $\tilde K$ (actually $K$) in $\phi_8(\DD_8)$. But $\phi_8(\DD_8)|\BB_4=\phi_1^2+\phi_3$ has $K$-fixed points of dimension 2.  Hence our copy of $\GL_8$ is double covered by a conjugate of $L(Q)$ and the  weight space $\nu_4$ of the restriction of $\phi_8(\DD_8)$   to $L(Q)$ lies in $Y$ and is fixed by $\tilde K$.   The group $N_{\BB_4}(\tilde K)/\tilde K\simeq\Z/2$    flips the highest and lowest weight spaces $\nu_{\pm 4}$. Since $\tilde K$ is a maximal connected reductive subgroup of $\BB_4$,   the stabilizer of $\nu_4$ is $\tilde K$ and any point of $\nu_4$ lies on a closed orbit. The slice representation of $\tilde K$ is $S^2(\C^4)+\theta_1$ which shows that the principal isotropy group $H$ of the action of $\BB_4$ (actually $\SO_9)$ is $SO_3\times\SO_3\times\SO_3$. It follows that $\dim \quot Y{\BB_4}=2$. Now $N_{\SO_9}(H)/H\simeq W(\DD_3)$, the Weyl group of $\DD_3$, where $V:=\phi_8(\DD_8)^H$ has dimension 5. One easily computes that the generators of $\O(V)^{W(\DD_3)}$ are of degree at most $5$. Then by the Luna-Richardson theorem \cite[3.2 Corollary]{Luna75} it follows that the invariants of $\O(Y)^{\SO_9}$ have generators in degree at most $5$, and then from our information about the Poincar\'e series it follows that $\O(Y)^{\SO_9}=\C[f_2,f_3]$ where $\deg f_i=i$, $i=2$, $3$.
\end{proof}

  \begin{remark} There is no representation of $\SO_9$ with principal isotropy group $H=\SO_3\times\SO_3\times\SO_3$ and slice representation $S^2(\C^4)+\theta_1$ of $K=\SO_6\times\SO_3$ which has homogeneous invariants $f_2$ and $f_3$ of degrees 2 and 3, respectively. The reason is that we would have a slice which is   an open   $K$-invariant subset of the  linear subspace $V=\C\cdot v+S^2(\C^4)$ where $K$ fixes $v$, and the restrictions of the $f_i$ to $V$ would have to be functions of $\C\cdot v$ alone since the invariant of $S^2(\C^4)$ is of degree 4. Thus $f_2$ and $f_3$ would be algebraically dependent, a contradiction to normality.
  \end{remark}
  
  \begin{remark}
  The generators $f_2$ and $f_3$ form a homogeneous regular sequence in $\O(Y)$, hence $\O(Y)$ is a free graded $\C[f_2,f_3]$-module \cite[Lemma 3.3]{Stanley79}. It follows that $\O(Y)$ is cofree, i.e., each module of covariants is free over $\C[f_2,f_3]$. Of course, we have the analogous result for $\O(X)$.
  \end{remark}

\providecommand{\bysame}{\leavevmode\hbox to3em{\hrulefill}\thinspace}
\providecommand{\MR}{\relax\ifhmode\unskip\space\fi MR }
\providecommand{\MRhref}[2]{%
  \href{http://www.ams.org/mathscinet-getitem?mr=#1}{#2}
}
\providecommand{\href}[2]{#2}

\end{document}